\documentclass[12pt, a4paper,reqno]{amsart}
\usepackage[a4paper, margin=4cm]{}
\usepackage{amssymb,amsfonts, amsthm, xfrac, amscd}
\usepackage{setspace}
\usepackage{breqn}

\usepackage{changes}
\usepackage{adjustbox}
\usepackage{color}
\usepackage{array}

\usepackage{color}
\usepackage[hidelinks]{hyperref}

\allowdisplaybreaks[4]

\newtheoremstyle{myremark}     {10pt}{10pt}{}{}{\bfseries}{.}{.5em}{}

\newtheorem{thm}{Theorem}[section]
\newtheorem{cor}[thm]{Corollary}
\newtheorem{lem}[thm]{Lemma}
\newtheorem{pro}[thm]{Proposition}
\theoremstyle{definition}
\newtheorem{defn}[thm]{Definition}
\newtheorem{exmp}[thm]{Example}
\theoremstyle{myremark}
\newtheorem{rem}[thm]{Remark}

\numberwithin{equation}{section}
\begin{document}

	\title[Bilateral gamma approximation]{Bilateral gamma approximation on Wiener Space}

	\author[Barman, Ichiba and  Vellaisamy]{Kalyan Barman, Tomoyuki Ichiba and  Palaniappan Vellaisamy}
	\address{\hskip-\parindent
		Kalyan Barman, Department of Mathematics, IIT Bombay,
		Powai - 400076, India.}
	
	\email{barmankalyan@math.iitb.ac.in}


		
		\address{\hskip-\parindent
			Tomoyuki Ichiba, Department of Statistics and Applied Probability, UC Santa Barbara, Santa Barbara, CA, 93106, USA.}
		\email{ichiba@pstat.ucsb.edu}
		
		\address{\hskip-\parindent
		Palaniappan Vellaisamy, Department of Statistics and Applied Probability, UC Santa Barbara, Santa Barbara, CA, 93106, USA.}
	\email{pvellais@ucsb.edu}
	
	\begin{abstract}
		This paper deals with bilateral-gamma ($BG$) approximation to functionals of an isonormal Gaussian process. We use Malliavin-Stein method to obtain the error bounds for the smooth Wasserstein distance. As by-products, the error bounds for variance-gamma ($VG$), Laplace, gamma and normal approximations are presented. Our approach is new in the sense that the Stein equation is based on integral operators rather than differential operators commonly used in the literature. Some of our bounds are sharper than the existing ones. For the approximation of a random element from the second Wiener chaos to a $BG$ distribution, the bounds are obtained in terms of their cumulants. Using this result, we show that a sequence of random variables (rvs) in the second Wiener chaos converges in distribution to a $BG$ rv if their cumulants of order two to six converge. As an application of our results, we consider an approximation of homogeneous sums of independent rvs to a $BG$ distribution, and mention some related limit theorems also. Finally, an approximation of a $U$-statistic to the $BG$ distribution is discussed.
	\end{abstract}

	\maketitle
	
\section{Introduction} \label{Intro}	

\noindent
Approximations related to limit theorems are fundamental issues in probability theory; see, for instance \cite{BU1}, \cite{Novak1}, \cite{k24}, \cite{k1} and \cite{k7}. The study of approximations related to limit theorems on Wiener space has been a keen interest of researchers in recent years (see \cite{MS0}, \cite{key2} and \cite{nourdin}). The first study in this direction is the central limit theorem due to Nualart and Peccati (see \cite{nup0}). They proved a necessary and sufficient condition for a sequence of random variables (rvs), belonging to a fixed Wiener chaos, converges in distribution to a Gaussian random variable (rv). This result is known as fourth moment theorem in the literature. 

\vskip 1ex
\noindent
A few years later, Nourdin and Peccati \cite{nourdin} combined Stein's method for normal approximation with Malliavin calculus on the Wiener space, called the Malliavin-Stein method, which provides an accuracy of approximation in the fourth moment theorem for functionals on Wiener space.

\vskip 1ex
\noindent 
The Malliavin-Stein method has also been extended to other probability distributions such as product normal \cite{MSPN}, Laplace \cite{gauntnew} and gamma \cite{nourdin1} distributions. In this direction, Eichelsbacher and Th$\ddot{\text{a}}$le \cite{key2} developed Malliavin-Stein method for variance-gamma ($VG$) approximation on the Wiener space. The Stein equation considered in \cite{key2} is a second-order differential equation. Using the estimates of the solution of the Stein equation, they obtain the bounds for $VG$ approximation in terms of Malliavin operators and norms of contractions. However, their bounds contain two indeterminate positive constants (see Equation (4.1) of \cite{key2}). Further, the authors establish a ``six moment theorem'' for the $VG$ approximation of double Wiener-It$\hat{\text{o}}$ integrals. Later, $VG$ approximations on the Wiener space are also discussed in \cite{rosen3,MSPN,MSLC0,rosen1}, in which the limiting distributions are represented as the difference of two centered gamma rvs or the product of two mean zero normal rvs. Also, Gaunt \cite{kk2} considered symmetric variance-gamma ($SVG$) approximation of double Wiener-It$\hat{\text{o}}$ integrals. Recently, Azmoodeh {\it et al.} \cite{MS001} established a six moment theorem for $VG$ approximation of double Wiener-It$\hat{\text{o}}$ integrals with optimal convergence rate. More recently, Gaunt \cite{gauntnew} obtains some new bounds for the solution of $VG$ Stein equation. As an application, the author obtains explicit error bounds for $VG$ approximation of double Wiener-It$\hat{\text{o}}$ integrals.

\vskip 1ex
\noindent
The bilateral-gamma ($BG$) distributions are the convolution of two gamma distributions, which have a rich distributional theory and applications (see \cite{bilateral0} and \cite{bilateral1}). This family contains several classical distributions as special or limiting cases, such as normal, product normal, gamma, $SVG$, $VG$, normal variance-mean mixture and Laplace distributions. The problem of $BG$ approximation on Wiener space is new and only some special or limiting cases such as gamma, normal, product-normal, Laplace, $VG$ and $SVG$ approximations have been considered so far in the literature. Therefore, it is of interest to study $BG$ approximation for functionals on Wiener space and obtain the error bounds so that the corresponding bounds and limit theorems follow for some special distributions mentioned above.

\vskip 1ex
\noindent
In this article, we combine Malliavin calculus and Stein's method to derive bounds for the $BG$ approximation of functionals of an isonormal Gaussian process. In particular, we derive bounds for the $BG$ approximation of the rvs belonging to the second Wiener chaos, in terms of their cumulants. Our approach of deriving bounds is also new as our Stein equation is based on integral operators, rather than differential operators commonly used in the literature. Also, this approach is rather simpler as can be seen for example in the case of $VG$ approximations. It is also shown that some of our bounds are sharper than the existing bounds. Further, we establish that a sequence of rvs in the second Wiener chaos converges in distribution to a $BG$ rv if their cumulants of order two to six converge to that of a $BG$ rv. We discuss also the approximations of homogeneous sums of independent rvs to a $BG$ distribution and show that the limiting distribution belongs to the class of $BG$ distributions. Finally, as an application, the convergence of a $U$-statistic to the $BG$ distribution is also presented.

\vskip 1ex
\noindent
The article is organized as follows. In Section \ref{nopre}, we introduce some notations and review some required preliminary results. Section \ref{probgdistribution} deals with $BG$ distribution, and its properties. A Stein equation for a $BG$ distribution, and its solution are obtained. In Section \ref{BSBGD}, we derive the error in approximation of a functionals of an isonormal Gaussian process to a $BG$ distribution. Also, for a sequence of rvs from the second Wiener chaos and its approximation to a $BG$ distribution, the error bounds are obtained in terms of their cumulants. As a special case, we mention the results for normal, gamma, $VG$, $SVG$, and Laplace distributions. A comparison of our results with the existing results is also presented. In Section \ref{applica}, we consider approximation problem of homogeneous sums of independent rvs to a $BG$ distribution. As corollaries, the limit theorems for $SVG$ and normal distributions are pointed out. As an example, an approximation of a $U$-statistic to the $BG$ distribution is considered.

\section{Notations and Preliminaries}\label{nopre}
\noindent
In this section, we introduce the notations and state some preliminaries required for this article.
\subsection{Function spaces and probability metrics}\label{PP2:FS}
Let $\mathbb{N}= \{1,2,\ldots,\infty \}$ be the set of all natural numbers and  $\mathbb{N}_{0}=\mathbb{N} \cup \{ 0\}$. Let $\mathcal{S}(\mathbb{R})$ be the Schwartz space defined by
$$\mathcal{S}(\mathbb{R}):=\left\{f\in C^\infty(\mathbb{R}): \lim_{|x|\rightarrow \infty} |x^m f^{(n)}(x)|=0, \text{ for all } m,n\in \mathbb{N}_{0} \right\},$$
\noindent
where $f^{(n)}$, $n\geq 1$, denotes henceforth the $n$th derivative of $f$ with $f^{(0)}=f$ and $C^\infty(\mathbb{R})$ is the space of infinitely differentiable functions on $\mathbb{R}$. Note that the Fourier transform (FT) on $\mathcal{S}(\mathbb{R})$ is automorphism. If $f\in \mathcal{S}(\mathbb{R})$, then its FT is given by $\widehat{f}(u):=\int_{\mathbb{R}}e^{-iux}f(x)dx,~~\text{  }u\in\mathbb{R} $ and its inverse FT is given by $f(x):=\frac{1}{2 \pi}\int_{\mathbb{R}}e^{iux}\widehat{f}(u)du,~~\text{  }x\in\mathbb{R} $, see \cite{stein}.

\vskip 2ex
\noindent
Next, we define smooth Waserstein distance (see \cite{k0} or \cite{MS001}). Consider the function space
\vfill
\vfill

\begin{align}\label{fs1}
\mathcal{W}_r = \left\{h:\mathbb{R}\to \mathbb{R} ~\big|~ h \mbox{ is $r$ times differentiable and }\|h^{(k)}\|\leq 1,~ k =0, 1,\ldots, r \right\},
\end{align} 
\noindent
where $\|h\|=\sup_{x\in\mathbb{R}}|h(x)|$. Then, for any two random variables $Y$ and $Z$,
\begin{equation}\label{smwdis}
\sup_{h \in \mathcal{W}_r}\left|\mathbb{E}[h(Y)]-\mathbb{E}[h(Z)]\right|:=d_{\mathcal{W}_r}(Y,Z)=d_{r}(Y,Z)~(\text{say})
\end{equation}
\noindent
is called the smooth Wasserstein distance. Also, when we choose the function spaces
\begin{align*}
	\mathcal{M}_K &= \left\{h:\mathbb{R}\to \mathbb{R}~\big|~h=\mathbf{1}_{(-\infty,x]},~x\in \mathbb{{R}} \right\},\\
	\mathcal{M}_W &= \left\{h:\mathbb{R}\to \mathbb{R}~\big|~h \text{ is 1-Lipschitz and }|h(x)-h(y)|\leq |x-y|  \right\},
\end{align*}


\noindent
$d_{\mathcal{M}_K}=d_K$ (say) and $d_{\mathcal{M}_W}=d_W$ (say) are called the Kolmogorov and Wasserstein distances, respectively (see \cite{k0}). It is known from Appendix A of \cite{k0}
\begin{align*}
d_{{r-1}}(Y,Z) \leq 3 \sqrt{2} \sqrt{d_{{r}}(Y,Z)}, ~r \geq 2,
\end{align*}
\noindent
and
\begin{align}\label{mre1}
d_{{1}}(Y,Z) \leq \left(3 \sqrt{2} \right)^{\sum_{i=1}^{r-1}\frac{1}{2^{i-1}}} \left(d_{{r}}(Y,Z) \right)^{\frac{1}{2^{r-1}}},~r\geq 2.
\end{align}

\noindent
Also, 
$$d_{K}(Y,Z) \leq c \sqrt{ d_{{1}}(Y,Z)},$$
\noindent
where $c$ is some positive constant. Using \eqref{mre1}, we get
$$d_{K}(Y,Z) \leq c\left(3 \sqrt{2} \right)^{{(1-2^{1-r})}} \left(d_{{r}}(Y,Z) \right)^{2^{-r}},~r\geq 2.$$
\noindent
In addition, the order relationship between $d_{r}$ and $d_W$ is 
\begin{align}\label{metrico1}
d_{{r}}(Y,Z) \leq d_{{1}}(Y,Z) \leq d_{W}(Y,Z),~r \geq 1. 
\end{align}

\subsection{Some basic results on Malliavin calculus}\label{mcal}
Here, we present some basic properties from Gaussian analysis
and Malliavin calculus, that will be used later. Let $\mathcal{H}$ be a real separable Hilbert space and $Y=\{Y(h): h\in \mathcal{H} \}$ be a centered Gaussian family defined on a probability space $(\Omega,\mathcal{F},\mathbb{P})$ such that $\mathbb{E}(Y(h))=0$ and $Cov(Y(h),Y(g))=\langle h,g \rangle_{\mathcal{H}}$. Note that $Y$ is called the isonormal Gaussian process over $\mathcal{H}$. Let $\mathcal{H}^{\otimes m}$ and $\mathcal{H}^{\odot m}$ denote respectively $m$th tensor product and $m$th symmetric tensor product of $\mathcal{H}$ with convention $\mathcal{H}^{\otimes 0}=\mathcal{H}^{\odot 0}=\mathbb{{R}}.$ For $q\geq 1$, the $q$th degree Hermite polynomial, denoted by $H_q(x)$, $x\in \mathbb{{R}}$, is defined as 
\begin{align}H_q(x)=(-1)^q e^{\frac{x^2}{2}} \frac{d^q}{dx^q}(e^{-\frac{x^2}{2}})=(-1)^{q} \frac{\phi ^{(q)}(x)}{\phi(x)}, 
\end{align}
\noindent
 where $\phi(x)$ denotes the density of the standard normal $ \mathcal{N}(0,1)$ distribution. Then, the $q$th Wiener chaos of $Y$, denoted by $\mathcal{H}_q$, defined as the closed linear subspace of $L^2 (\Omega,\mathcal{F},\mathbb{P}):=L^2(\Omega)$, which is generated by the family $\{H_q(Y(h)) : h \in \mathcal{H}, \|h\|_{\mathcal{H}} =1 \}$. We write by convention $\mathcal{H}_0=\mathbb{{R}}.$ For $q\geq 1$, the mapping $I_q(h^{\otimes q})= H_{q}(Y(h))$ can be extended to a linear isometry between $\mathcal{H}^{\odot q}$ and the $q$th Wiener chaos $\mathcal{H}_q.$ For $q=0$, we write $I_0(c)=c$, $c\in \mathbb{{R}}$. It is well-known that $L^2(\Omega)$ can be decomposed into the infinite orthogonal sum of the spaces $\mathcal{H}_q$. Hence for any $G\in L^2(\Omega)$, we get by Stroock formula
\begin{align}\label{expansion}
	G=\sum_{q=0}^{\infty}I_{q}(h^{\odot q})=\mathbb{E}(G)+\sum_{q=1}^{\infty}H_q(Y(h)),
\end{align}
\noindent
where $h^{\odot q} \in \mathcal{H}^{\odot q},~q\geq 0,$ with $h^{\odot 0}=\mathbb{E}(G).$ 
\vskip 2ex
\noindent
Let $1 \leq p \leq q$ be integers and let $\{e_j \}_{j \geq 1}$ be an orthonormal basis of $\mathcal{H}$. Given integers $j_1,\ldots,j_p,k_1,\ldots,k_q$ and $r=0,1,\ldots,p,$ the $r$th contraction of the two tensor products $e_{j_1} \otimes \ldots \otimes e_{j_p}$ and $e_{k_1} \otimes \ldots \otimes e_{k_q},$ is defined as follows.  
$$(e_{j_1} \otimes \ldots \otimes e_{j_p}) \otimes_0 (e_{k_1} \otimes \ldots \otimes e_{k_q})=e_{j_1} \otimes \ldots e_{j_p} \otimes e_{k_1} \otimes \ldots \otimes e_{k_q},$$
\noindent
and, for $r=1,2,\ldots,p,$
$$(e_{j_1} \otimes \ldots \otimes e_{j_p}) \otimes_r (e_{k_1} \otimes \ldots \otimes e_{k_q})=\bigg[\prod_{l=1}^{r} \langle e_{jl},e_{kl} \rangle_{\mathcal{H}} \bigg]e_{j_{r+1}} \otimes \ldots e_{j_p} \otimes e_{k_{r+1}} \otimes \ldots \otimes e_{k_q}.$$
\noindent
If $r=q=p,$ then
$$(e_{j_1} \otimes \ldots  \otimes e_{j_p}) \otimes_p (e_{k_1} \otimes \ldots \otimes e_{k_q})=\prod_{l=1}^{p} \langle e_{jl},e_{kl} \rangle_{\mathcal{H}}=\langle e_{j_1} \otimes \ldots \otimes e_{j_p},  e_{k_1} \otimes \ldots \otimes e_{k_p}\rangle_{\mathcal{H}^{\otimes p}} .$$
\noindent
Next, let $f\in \mathcal{H}^{\otimes p}$ and $g \in \mathcal{H}^{\otimes q}$, then for every $r=0,1,\ldots,p\wedge q,$ the contraction of $f$ and $g$ of order $r$ is the element of $\mathcal{H}^{\otimes (p+q-2r)}$, defined by
$f \otimes_0 g=f \otimes g$, and for $p=q$, $f \otimes_p g= \langle f,g \rangle _{ \mathcal{H}^{\otimes p}},$ and for $1\leq r \leq (p-1),$

$$f \otimes_r g= \sum_{i_1,\ldots,i_r=1}^{\infty} \langle f, e_{i_1} \otimes \ldots \otimes e_{i_r} \rangle_{\mathcal{H} ^{\otimes r}} \otimes \langle g, e_{i_1} \otimes \ldots \otimes e_{i_r} \rangle_{\mathcal{H} ^{\otimes r}}. $$

\noindent
Note that $f \otimes_r g$ does not depend on the basis and also not necessarily symmetric. The symmetric version is denoted by $f \tilde{\otimes}_r g \in \mathcal{H}^{\odot (p+q-2r)}$. Also, if $f \in \mathcal{H}^{\odot p}$ and $g \in \mathcal{H}^{\odot q},$ and $1\leq q \leq p$, then
\begin{align}\label{prductfor1}
I_p(f)I_q(g)=\sum_{r=0}^{p \wedge q} \binom{p}{r} \binom{q}{r} I_{p+q-2r}(f \tilde{\otimes}_r g ),
\end{align} 
\noindent
and

\begin{align}\label{prductfor2}
\mathbb{E} \bigg(I_p(f) I_q (g) \bigg)&=\begin{cases}
p! \langle f, g\rangle_{\mathcal{H}^{ \otimes p} },& \text{ if }q=p\\
0, &\text{ otherwise}.
\end{cases} 
 \end{align}

\vskip 2ex
\noindent
  Let $Y$ be an isonormal Gaussian process. Let $g: \mathbb{{R}}^n \to \mathbb{{R}}$ be a $C^\infty$-function such that its partial derivatives have polynomial growth. Then a rv $g$ of the form $G=(Y(h_1),Y(h_2),\ldots,Y(h_n))$
 with $n\geq 1$, $h_1,h_2,\ldots,h_n \in \mathcal{H}$, is called a smooth rv and let $\mathcal{S}$ be the set of all smooth rvs. The Malliavin derivative of $G$ with respect to $Y$ is the element of $L^{2}(\Omega, \mathcal{H})$ defined as
 $$DG=\sum_{i=1}^{n}\frac{\partial }{\partial y_i}g(Y(h_1),\ldots,Y(h_n)) h_i,$$
 \noindent
  and the $m$th derivative ($m\geq 2$) $D^m G$ is an element of $L^2 (\Omega, \mathcal{H}^{\odot m})$, defined by
 $$D^m G=\sum_{i_1,\ldots,i_m=1}^{m}\frac{\partial ^m}{\partial y_{i_1},\ldots,\partial y_{i_m} }g(Y(h_1),\ldots,Y(h_m))h_{i_1} \otimes \ldots \otimes h_{i_m}.$$
 \noindent
 For $m\geq 1$ and $p\geq 1$, let $\mathbb{D}^{m,p}$ denotes the closure of $\mathcal{S}$ with respect to the norm defined as
 $$\|G \|^{p}_{m,p}= \mathbb{E}[|G|^p]+\sum_{i=1}^{m}\mathbb{E}[\|D^{i}G\|^{p}_{\mathcal{H}^{\otimes i}}  ],$$
 \noindent
 and $\mathbb{D}^{\infty}:= \cap_{m\geq 1}\cap_{p\geq 1}\mathbb{D}^{m,p}$. That is, $\mathbb{D}^{m,p}$ is the domain of $D^m$ in $L^p(\Omega)$. Moreover, for any $q\geq p \geq 2$ and $n \geq m$, $\mathbb{D}^{n,q} \subset \mathbb{D}^{m,p}$. 
 
 \vskip 2ex
 \noindent
 For $p\geq 1$, the adjoint of the operator $D^p$, called the $p$-th divergence operator, is denoted by $\delta^p$. The domain of $\delta^p$, denoted by $Dom(\delta^{p})$, is defined as 
 \begin{align*}
 	Dom(\delta^{p})=\Big\{u\in L^{2}(\Omega, \mathcal{H}^{\otimes p} ) & ~\Big|~  |\mathbb{E}[\langle D^{p}G, u  \rangle_{\mathcal{H} ^{\otimes p}} ]| \\
 	& \leq c_{u} \sqrt{\mathbb{E} (G^2)},~~\text{for some constant $c_u>0$} \Big\},
 \end{align*}
 \noindent
 see \cite{MS1}.  Then the $p$th divergence operator $\delta^{p}:Dom(\delta^{p}) \subset L^{2}(\Omega, \mathcal{H}^{\otimes p} ) \to L^{2}(\Omega)$ is defined, for each $u\in Dom (\delta^p),$ by
\begin{align}\label{dualfo}
\mathbb{E}(G\delta^{p}(u))=\mathbb{E}(\langle D^{p}G,u \rangle_{\mathcal{H}^{\otimes p}}),~G\in \mathcal{S},
\end{align}
\noindent
which is sometimes called the duality formula. 

 \noindent
 It is well-known that the Ornstein-Uhlenbeck (O-U) semigroup (see Section 2.8 of \cite{nourdin}) is defined by

 \begin{align}
 	P_t(G)=\sum_{q=0}^{\infty}e^{-qt}J_{q}(G),~t\geq 0,
 \end{align}
 \noindent
 where $J_{q}(G)=\text{Proj}(G|\mathcal{H}_q)=I_q(h^{\odot q})$ ($q\geq 0$) for any $G$ as in \eqref{expansion}. Also, the infinitesimal generator of the O-U semigroup is given by $L=\sum_{q=0}^{\infty}-qJ_q,$ where the domain of $L$ is
 $Dom(L)=\{F\in L^{2}(\Omega) : \sum_{q=1}^{\infty}q^2 \mathbb{E} (J{_q}(G))^2 < \infty   \}.$ It is known from the Proposition 2.8.8 of \cite{nourdin} that $\delta (DG)=-LG.$ Also, $D(f(G))=f^{\prime}(G)DG.$

 \noindent
 For any $G\in L^{2}(\Omega),$ the operator $L^{-1}G=\sum_{q=1}^{\infty}-\frac{1}{q}J_{q}(G),$ is called pseudo-inverse of $L$, and $L^{-1}G \in \text{Dom}L,$ and $LL^{-1}G=G-\mathbb{E}(G).$
 
 \vskip 1ex
 \noindent
 Let $f:\mathbb{{R}} \to \mathbb{{R}}$ be a differentiable function having bounded derivative. Then, for any $G,H\in \mathbb{D}^{1,2}$, the Malliavin integration by parts formula (see Theorem 2.9.1 of \cite{nourdin}) gives
 \begin{align}
 	\mathbb{E} \left(Gf(H)   \right)&=\mathbb{E}(G)\mathbb{E}(f(H))+\mathbb{E}\left(f^{\prime} (H) \langle DH,-DL^{-1} G\rangle_{\mathcal{H}} \right).\label{intbypfo}
 \end{align}
 \noindent
 When $H=G$, we get
 \begin{align}
 \nonumber\mathbb{E} \left(Gf(G)   \right)&=\mathbb{E}(G)\mathbb{E}(f(G))+\mathbb{E}\left(f^{\prime} (G) \langle DG,-DL^{-1} G\rangle_{\mathcal{H}} \right)\\
 &=\mathbb{E}(G)\mathbb{E}(f(G))+\mathbb{E}\left(f^{\prime} (G) \Gamma_{2}(G) \right),\label{ibp00}
 \end{align}
 \noindent
 where, for a rv $G\in\mathbb{D}^{\infty}$, $\Gamma_1(G)=G$, and for every $m\geq2,$
\begin{align}\label{gamop}
\Gamma_{m}(G)=\langle DG,-DL^{-1} \Gamma_{m-1}(G)\rangle_{\mathcal{H}} \in \mathbb{D}^{\infty},	
\end{align}
\noindent
are called $\Gamma$-operators. The relation \eqref{gamop} also holds under weak assumptions on the regularity of $G$, see Theorem 4.3 of \cite{nourdin0}. For $G\in \mathbb{D}^{1,2}$, it follows that $\Gamma_{2}(G) \in L^1(\Omega)$ and $Var(G)=\mathbb{E}(\Gamma_{2}(G))$, and for $G\in \mathbb{D}^{1,4}$, it follows that $\Gamma_{2}(G) \in L^2(\Omega)$.


\vskip 1ex
\noindent
Next we discuss relationship between $\Gamma$-operators and cumulants (see \cite{nourdin0} for more detail). Let $G$ be a rv with cf $\phi_{G}(t)=\mathbb{E} (e^{itG})$, $t\in \mathbb{{R}}$, and $\mathbb{E}(G^m)<\infty$, for $m\geq1$. Then its $m$-th cumulant is defined as
\begin{align}
\kappa_{m}(G):=(-i)^{m}\left[\frac{d^{m}}{dt^{m}}\log\phi_{G}(t)\right]_{t=0},~m\geq 1. 
\end{align}
\noindent
It is known that $\kappa_1(G)=\mathbb{E}(G)$, $\kappa_{2}(G)=Var(G)$. Also, for a rv $G\in \mathbb{D}^{\infty}$ (see \cite{MS0}), 
\begin{align}\label{CuGamma}
	\kappa_{m}(G)=(m-1)! \mathbb{E}(\Gamma_{m}(G)), ~m\geq 1.
\end{align}

\subsubsection{Some useful properties of the second Wiener chaos} For a rv $G=I_2(f)$, where $f\in\mathcal{H}^ {\odot 2}$, the Hilbert-Schmidt operator is defined (see \cite[p.42]{nourdin}) as 
$$A_f : \mathcal{H} \to \mathcal{H};~A_f(g) \to f \otimes_1 g ~\text{(1-contraction of $g$)},$$
\noindent
with eigen values $\{\lambda_{f,j} \}_{j\geq 1}$ and eigen vectors $\{e_{f,j} \}_{j\geq 1}.$ The sequence of auxiliary kernels $\{f \otimes_{1}^{(p)} f : p\geq 1  \} \subset \mathcal{H}^{\odot 2}$ are recursively defined by :
\begin{align}
f \otimes_1^{(1)}f=f;~f \otimes_1^{(p)}f=(f \otimes_1^{(p-1)}f)\otimes_1 f, 
\end{align}
\noindent
 for $p\geq 2$. In particular, $f \otimes_1^{(2)}f=(f \otimes_1^{(1)}f)\otimes_1 f=f \otimes_1 f.$ Also, for $m,p\geq 1,$
 \begin{align}\label{contract1}
 	\langle f\otimes_{1}^{(m)}f,f\otimes_{1}^{(p)}f\rangle_{\mathcal{H}^{\odot 2}}=\langle f\otimes_{1}^{(m+p-1)}f,f\otimes_{1}^{(1)}f\rangle_{\mathcal{H}^{\odot 2}}=\langle f\otimes_{1}^{(m+p-1)}f,f\rangle_{\mathcal{H}^{\odot 2}}.
 \end{align}


\begin{pro}\cite[p.43 ]{nourdin}
	Let $G=I_2(f)$ with $f\in \mathcal{H}^{\odot 2}$ be a rv in the second Wiener chaos. Then
	\begin{enumerate}
		\item [(i)] The rv $G$ admits the representation $G = \sum_{j=1}^{\infty}\lambda_{f,j} (Z_{j}^{2} -1),$ where $Z_j \sim \mathcal{N}(0,1).$ The random series converges in $L^2 (\Omega).$
		\item [(ii)] For every $p\geq 2,$ the $p$-th cumulant of $G$ is given by
		\begin{align}\label{sechaoscum}
		\kappa_p(G)=2^{p-1}(p-1)!\sum_{j=1}^{\infty}\lambda_{f,j}^{p}=2^{p-1}(p-1)!\langle f\otimes_1^{(p-1)} f,f \rangle_{\mathcal{H}^ {\odot 2}}.
		\end{align}
	\end{enumerate}
\end{pro}

\subsection{Malliavin-Stein approach on Wiener space}\label{SM}
Finally, we discuss the combination of Stein's method with Malliavin calculus on the Wiener space. In general, Stein's method is based on the fact that, any real-valued random variable $Z$ follows a distribution $F_Z$ if and only if there exists an operator $\mathcal{A}$ (also called Stein operator) such that 
\begin{equation*}
	\mathbb{E}\left(\mathcal{A}f(Z) \right)=0,~~\text{for every}~~f \in \mathcal{F},
\end{equation*}
\noindent
where $\mathcal{F} $ is a suitably chosen function space. This characterization leads us to the Stein equation
\begin{equation}\label{normal2}
	\mathcal{A}f(x)=h(x)-\mathbb{E}h(Z), ~ x \in \mathbb{R},
\end{equation}
\noindent
where $h$ is a real-valued test function. Replacing $x$ with a rv $Y$ belongs to a fixed Wiener chaos, and taking expectations on both sides of  (\ref{normal2}), we get
\begin{equation}\label{normal3}
\mathbb{E}h(Y)-\mathbb{E}h(Z)=\mathbb{E}\left(\mathcal{A}f(Y) \right).
\end{equation}
\noindent 
The equality (\ref{normal3}) plays a crucial role in this approach. For a real valued test function $h$, the problem of bounding the quantity $|\mathbb{E}h(Y)-\mathbb{E}h(Z)|$ relies on the bounds for the solution of (\ref{normal2}) and behavior of $Y$. Recall that (see \eqref{smwdis})
$$d_r(Y,Z) =\sup_{h \in \mathcal{W}_r} \left|\mathbb{E}(h(Y))-\mathbb{E}(h(Z))\right|.$$
\noindent
From \eqref{normal3}, it follows that

$$d_r(Y,Z) \leq \sup_{f \in \mathcal{F}(\mathcal{W}_r)}\left| (\mathbb{E}\mathcal{A}f(Y)) \right|,$$
\noindent
 where $\mathcal{F}(\mathcal{W}_r)=\{f_h: h\in \mathcal{W}_r  \}$ is the set of solutions that satisfy \eqref{normal3}. For more details on Malliavin-Stein approach, see the recent review paper \cite{MS0} and the references therein.

\vskip 1ex
\noindent
In particular, if $q\geq 2$ be an integer and $ G_n=I_q(f_{n}),$ $n\geq 1$, such that $f_{n} \in \mathcal{H}^{\odot q}$ be a sequence of rvs belonging to the $q$-th Wiener chaos such that $\mathbb{E}(G_n^2)=1.$ Then, the fourth moment theorem (see Theorem 1 of \cite{nup0}) states that  $G_n \overset{\mathcal{L}}{\to} Z$ if and only if $\mathbb{E}(G_n^4) \to 3$, where $Z\sim \mathcal{N}(0,1)$ and the notation $\overset{\mathcal{L}}{\to}$ denotes the convergence in distribution. Using the Malliavin-Stein approach, Nourdin and Peccati proved (see Theorem 5.2.6 of \cite{nourdin}) that the upper bound of the Kolmogorov distance between $G_n$ and $Z$ is 
\begin{align}\label{fmth}
d_K(G_n,Z) \leq \sqrt{\frac{q-1}{3q}\left(\mathbb{E}(G_n^4)-3  \right)  }=\sqrt{\frac{q-1}{3q}\kappa_{4}(G_n)},
\end{align}
\noindent
where $\kappa_{4}(G_n)$ is the fourth cumulant of $G_n$. From \eqref{fmth}, if $\kappa_{4}(G_n) \to 0$ (or equivalently if $\mathbb{E}(G_n^4) \to \mathbb{E}(Z^4)=3$), then $d_K(G_n,Z) \to 0 \text{ as } n\to \infty,$ which implies that the rv $G_n$ converges to a normal $\mathcal{N}(0,1)$ distribution.

\section{Some basic results for bilateral gamma distribution}\label{probgdistribution}
\noindent
We start with definition and related results.
\subsection{Definition and properties}\label{prebg}
Here, we define $BG$ distributions and review some of their properties. The $BG$ distributions have a rich distributional theory, and contains several classical distributions as special or limiting cases, such as the normal, Laplace, $VG$, normal variance-mean mixtures and etc. (see \cite{bilateral0} and \cite{bilateral1}). Let $X_i,$ $i=1,2,$ be two independent gamma $Ga(\alpha_i,p_i)$ rvs with density
$$f_{i}(x)= \frac{\alpha_i^{p_i}}{\Gamma (p_i)}e^{-\alpha_i x}x^{p_i -1}, ~x>0,~i=1,2.$$
\noindent
Then the distribution of the rv $X=X_1-X_2$ is called the $BG$ distribution and

\noindent
we denote it by $BG(\alpha_1,p_1,\alpha_2,p_2).$ The characteristic function (cf) of a $BG$ distribution is 
\begin{align}
\phi_{bg}(z)&=\left( \frac{\alpha_1}{\alpha_1 -iz} \right)^{p_1}\left( \frac{\alpha_2}{\alpha_2 +iz} \right)^{p_2}\label{bicf0}\\
&=\exp\left(\int_{\mathbb{R}}(e^{izu}-1)\nu_{bg}(du)  \right),~~z\in\mathbb{R},\label{e1}
\end{align}
\noindent
where $\nu_{bg}$ is the L\'evy measure given by
\begin{align}\label{bglevymeasure}
\nu_{bg}(du)=\left(\frac{p_1  }{u}e^{-\alpha_1u}\mathbf{1}_{(0,\infty)}(u)-\frac{p_2  }{u}e^{-\alpha_{2}|u|}\mathbf{1}_{(-\infty,0)}(u)\right)du.
\end{align}

\noindent
Next, we list some special and limiting distributions of the $BG$ family (see \cite{bilateral0}).
\begin{enumerate}
	\item[(i)] When $p_1=p_2=p$, then $BG(\alpha_1,p,\alpha_2,p)$ is the $VG$ distribution, denoted by $VG(\alpha_1,\alpha_2,p)$.
	\item [(ii)] When $\alpha_1=\alpha_2=\alpha$ and $p_1=p_2=p$, then $BG(\alpha,p,\alpha,p)$ follows $SVG$ distribution, denoted by $SVG(\alpha,p)$.
	\item [(iii)] The $SVG(\alpha,1)$ is called the Laplace distribution, which we denote by $La(\alpha)$.
	
	\item [(iv)] The limiting case as $p\to \infty$,  the $SVG( \sqrt{2p}/\alpha,p)$ converges in law to a $\mathcal{N}(0,\alpha^2)$ distribution.
	
	\item [(v)] The limiting case as $\alpha_2 \to \infty$, the $BG(\alpha_1,p_1,\alpha_2,p_2)$ converges in law to

	\noindent
	$Ga(\alpha_1,p_1)$ distribution.
\end{enumerate}
\noindent
Note that the $BG$ distributions are infinitely divisible and self-decomposable, see \cite{bilateral0}. Also, it is known that (see \cite{bilateral0}), if $Y_1 \sim BG(\alpha_1,p_1,\alpha_2,p_2)$ and $Y_2 \sim BG(\alpha_1,q_1,\alpha_2,q_2)$, then $Y_1+Y_2 \sim BG(\alpha_1,p_1+q_1,\alpha_2,p_2+q_2),$ that is, when the scale parameters of the gamma distributions are the same, the $BG$ distributions are stable under convolution. Let now $X\sim BG(\alpha_1,p_1,\alpha_2,p_2)$. The $j$th cumulant (see \cite{bilateral0}) is given by
\begin{align}
\kappa_{j}(X)&=(j-1)!\left(\frac{p_1}{\alpha_1^j}+(-1)^j \frac{p_2}{\alpha_2^j}  \right),~j\geq 1. \label{cuforbg}
\end{align}
\noindent
In particular, the first two cumulants (that is, mean and variance) of $X$ are 
$\kappa_1(X)=\mathbb{E}(X)= \left(\frac{p_1}{\alpha_1}-\frac{p_2}{\alpha_2}\right) \text{ and } \kappa_{2}(X)=\text{Var}(X)=\left(\frac{p_1}{\alpha_1^{2}}+\frac{p_2}{\alpha_{2}^{2}}\right).$ Also, from the definition of $BG$ rv $X$ as $X=X_1-X_2$, where $X_1,$ $X_2$ are independent gamma $Ga(\alpha_i,p_i)$, $i=1,2,$ rvs, we have
\begin{align}
\nonumber \mathbb{E}(X^k )&=\sum_{j=0}^{k}\binom{k}{j}(-1)^j\mathbb{E}(X_1^{k-j})\mathbb{E}(X_2^{j})\\
&=\sum_{j=0}^{k}\binom{k}{j}(-1)^j \frac{1}{ \alpha_{1}^{k-j} \alpha_2^j} \frac{\Gamma (p_1+k-j)  \Gamma(p_2+j) }{\Gamma p_1 \Gamma p_2 },
\end{align}
\noindent
since $\mathbb{E}(X_i^n)=\frac{1}{\alpha_i^{n}} \frac{\Gamma(p_i+n)}{\Gamma(p_i)},$ $i=1,2,$ for $n\geq1$, see \cite{gamdiff}.

\subsection{A Stein equation for a $BG$ distribution}\label{SEBGD}
In this section, we derive a Stein equation for the $BG$ distribution and solve it using the semigroup approach. Recently, a Stein equation for a $BG$ distribution in the form of second order differential equation is given in \cite{gamdiff}. Their approach is based on the density of the $BG$ distributions, which is quite involved. However, we derive here a Stein identity in the integral form involving the L\'evy measure. 
\begin{pro}\label{th1}
	Let $X \sim BG(\alpha_1,p_1,\alpha_2,p_2).$ Then,
	\begin{equation}\label{PP2:StenIdTSD}
	\mathbb{E}\left(Xf(X)-\displaystyle\int_{\mathbb{R}}f(X+u)u\nu_{bg}(du) \right)=0,~~f\in\mathcal{S}(\mathbb{R}),  
	\end{equation}
	\noindent
	where $\nu_{bg}$ is the L\'evy measure given in \eqref{bglevymeasure}.	
\end{pro}

\begin{proof}
 Taking logarithms on both sides of (\ref{e1}), and differentiating with respect to $z$, we have
	\begin{equation}\label{PP2:e4}
	\phi^{\prime}_{bg}(z)=i\phi_{bg}(z)\int_{\mathbb{R}}ue^{izu}\nu_{bg}(du).
	\end{equation}
	
	\noindent
	Let $h_X(x)$ denote density of $X$. Then
	\begin{equation}\label{PP2:e5}
	\phi_{bg}(z)=\displaystyle\int_{\mathbb{R}}e^{izx}h_X(x)dx~~\text{and}~~\phi^{\prime}_{bg}(z)=i\displaystyle\int_{\mathbb{R}}xe^{izx}h_X(x)dx.
	\end{equation}
	
	\noindent
	Substituting \eqref{PP2:e5} into  \eqref{PP2:e4} and rearranging the integrals, we have
	\begin{align}	
	\displaystyle\int_{\mathbb{R}}xe^{izx}h_X(x)dx-\phi_{bg}(z)\int_{\mathbb{R}}ue^{izu}\nu_{bg}(du)=0\label{PP2:e6}
	\end{align}
	\noindent
	The second integral of \eqref{PP2:e6} can be written as
	\begin{align}
	\nonumber \left(\int_{\mathbb R}ue^{izu}\nu_{bg}(du)\right)\phi_{bg}(z)&=\int_{\mathbb R}\int_{\mathbb R}ue^{izu}e^{izx}h_X(x)dx\nu_{bg}(du)\\
	\nonumber &=\int_{\mathbb R}\int_{\mathbb R}ue^{iz(u+x)}\nu_{bg}(du)h_X(x)dx\\
	\nonumber &=\int_{\mathbb R}\int_{\mathbb R}ue^{izy}\nu_{bg}(du)h_X(y-u)dy\\
	&=\int_{\mathbb R}e^{izx}\int_{\mathbb R}uh_X(x-u)\nu_{bg}(du)dx.\label{PP2:e7}
	\end{align}
	\noindent
	Substituting \eqref{PP2:e7} into \eqref{PP2:e6}, we have
	\begin{align}
	\nonumber	0&=\displaystyle\int_{\mathbb{R}}xe^{izx}h_X(x)dx-\int_{\mathbb R}e^{izx}\int_{\mathbb R}uh_{X}(x-u)\nu_{bg}(du)dx\\
	&=\displaystyle\int_{\mathbb{R}}e^{izx} \left( xh_X(x)-\int_{\mathbb R}uh_X(x-u)\nu_{bg}(du) \right)dx. \label{PP2:e8}
	\end{align}
	
	\noindent
	Applying FT to \eqref{PP2:e8}, multiplying with $f\in \mathcal{S}(\mathbb{R}),$ and integrating over $\mathbb{R},$ we get

	\begin{align}\label{PP2:e9}
	\displaystyle\int_{\mathbb{R}}f(x) \left( xh_X(x)-\int_{\mathbb R}uh_X(x-u)\nu_{bg}(du) \right)dx=0.
	\end{align}
	
	\noindent
	The second integral of \eqref{PP2:e9} can be seen as

	\begin{align}
	\nonumber \int_{\mathbb R}\int_{\mathbb R}uf(x)h_X(x-u)\nu_{bg}(du)dx&=\int_{\mathbb R}\int_{\mathbb R}uf(y+u)h_X(y)\nu_{bg}(du)dy\\ 
	&=\mathbb{E}\left(\int_{\mathbb R}f(X+u)u\nu_{bg}(du)\right).\label{PP2:e10}
	\end{align}
	\noindent
	Substituting \eqref{PP2:e10} in \eqref{PP2:e9}, we have
	\begin{align*}
	\mathbb{E}\left(Xf(X)-\displaystyle\int_{\mathbb{R}}f(X+u)u\nu_{bg}(du) \right)=0.
	\end{align*}
	\noindent
     This proves the result.
\end{proof}
\noindent
Note that for any $f \in \mathcal{S}(\mathbb{R})$,
\begin{align}\label{bisop}
	\mathcal{A}f(x):=-xf(x)+\displaystyle\int_{\mathbb{R}}f(x+u)u\nu_{bg}(du),
\end{align}
\noindent 
is a Stein operator for $BG$ distribution. Also, a Stein equation for $BG(\alpha_1,p_1,\alpha_2,p_2)$ distribution is given by
\begin{equation}\label{PP2:e15}
\mathcal{A}f(x)=-xf(x)+\displaystyle\int_{\mathbb{R}}f(x+u)u\nu_{bg}(du)= h(x)-\mathbb{E}(h(X)),
\end{equation}
\noindent
where $h\in \mathcal{M}$, a class of test functions. We now apply the semigroup approach to find $f_h$ that satisfies the integral equation \eqref{PP2:e15} for a given $h\in \mathcal{M}$. This approach for solving Stein equations is developed in Barbour \cite{k8} and Arras and Houdr\'e \cite{k0} generalized it for infinitely divisible distributions with the finite first moment. Henceforth, we consider $\mathcal{F}=\overline{\mathcal{S}}(\mathbb{R})$, the closure of $\mathcal{S}(\mathbb{{R}})$. Following Barbour's approach \cite{k8}, we choose a family of operators $(P_{t})_{t\geq0}$, for all $x\in\mathbb{R}$, as

\begin{equation}\label{PP2:e16}
P_{t}f(x):=\frac{1}{2\pi}\int_{\mathbb{R}}\widehat{f}(y)e^{iyxe^{-t} }\phi_t(y)dy, ~~f\in\mathcal{F},
\end{equation}
\noindent 
where $\hat{f}$ is the FT of $f$ and
\begin{align}\label{nPP2:a15}
\phi_t(y):=\frac{\phi_{bg}(y)}{\phi_{bg}(e^{-t}y)},~y\in \mathbb{{R}}.
\end{align}
\noindent
Since $BG$ distributions are self-decomposable (see \cite[p.263]{bilateral0}), $\phi_{t}$ is also a cf \cite[p.90]{sato} of some rv $X_{(t)}$, say. That is, for all $y\in \mathbb{R},$ and $t\geq 0,$
\begin{align}\label{PP2:a15}
\phi_t(y)=\displaystyle\int_{\mathbb{R}}e^{iy u}F_{X_{(t)}}(du),
\end{align}
\noindent
where $F_{X_{(t)}}$ is the law of $X_{(t)}$. Using \eqref{PP2:a15}, we get
\begin{align}
\nonumber  P_tf(x)&=\frac{1}{2\pi}\int_{\mathbb R}\int_{\mathbb{R}}\widehat{f}(y)e^{iy xe^{-t}}e^{iy u}F_{X_{(t)}}(du)dy\\
\nonumber  &=\frac{1}{2\pi}\int_{\mathbb R}\int_{\mathbb{R}}\widehat{f}(y)e^{iy (u+xe^{-t})}F_{X_{(t)}}(du)dy\\
&=\displaystyle\int_{\mathbb{R}}f(u+xe^{-t})F_{X_{(t)}}(du),\label{PP2:a17}
\end{align}
\noindent
where the last step follows by inverse FT (see Section \ref{PP2:FS}).

\begin{pro}\label{PP2:proSem}
	Let $(P_{t})_{t\geq0}$ be a family of operators defined in \eqref{PP2:e16}. Then
	\begin{itemize}
		\item [(i)] $(P_{t})_{t\geq 0}$ is a $\mathbb{C}_0$-semigroup on $\mathcal{F}$.
		\item [(ii)] Its generator $L$ is given by
		\begin{align}
		\nonumber Lf(x)&=-xf^{\prime}(x)+\displaystyle\int_{\mathbb R}f^{\prime}(x+u)u\nu_{bg}(du),~~f\in\mathcal{S}(\mathbb{R})\\
		\label{e7}&=\mathcal{A}f^{\prime}(x),
		\end{align}	
		
		\noindent
		where $\mathcal{A}$ is defined in \eqref{bisop}.
	\end{itemize}
\end{pro}
\begin{proof}
	(i) First note that $\lim_{t\to 0} \phi_{t}(y)=1$ and $\lim_{t\to\infty} \phi_{t}(y)=\phi_{bg}(y).$ Then it can be shown that $P_{0}f(x)=f(x) ~~\text{and}~~ \lim_{t\to\infty}P_{t}(f)(x)=\displaystyle\int_{\mathbb R}f(x)F_{X}(dx)=\mathbb{E}(f(X))$. 
	\noindent
	Now, for any $s,t\geq0$, we have
	\begin{align}\label{PP2:a22}
	\phi_{t+s}(y)=\frac{\phi_{bg}(y)}{\phi_{bg}(e^{-(t+s)}y)}
	=\frac{\phi_{bg}(y)}{\phi_{bg}(e^{-s}y)}\frac{\phi_{bg}(e^{-s}y)}{\phi_{bg}(e^{-(t+s)}y)}
	=\phi_{s}(y)\phi_{t}(ye^{-s}).
	\end{align}
	\noindent
	So, from \eqref{PP2:e16}, we have
	\begin{align}
	\nonumber	P_{t+s}(f)(x)&=\frac{1}{2\pi}\int_{\mathbb R}\widehat{f}(y)e^{iy xe^{-(t+s)}}\phi_{t+s}(y)dy\\
	\label{PP2:a022}	&=\frac{1}{2\pi}\int_{\mathbb R}\widehat{f}(y)e^{iy xe^{-(t+s)}}\phi_{s}(y)\phi_{t}(ye^{-s}) dy.
	\end{align}
	\noindent
	Let $\delta$ denote the Dirac $\delta$-function defined by
	$$\delta(x-a)=\frac{1}{2 \pi} \displaystyle\int_{\mathbb{R}}e^{i(x-a)y} dy, ~x,a\in\mathbb{{R}}.$$
	\noindent
	Then
	\begin{align*}
P_{t}(P_{s}(f))(x)&=\frac{1}{2\pi}\int_{\mathbb R}\widehat{P_{s}}(f)(y)e^{iy xe^{-t}}\phi_{t}(y)dy\\
	&=\frac{1}{2\pi}\int_{\mathbb R}\left(\int_{\mathbb{R}}e^{-ivy}P_{s}(f)(v)dv  \right)e^{iy xe^{-t}}\phi_{t}(y)dy\\
	&=\frac{1}{(2\pi)^{2}}\int_{\mathbb R}\left(\int_{\mathbb{R}}e^{-ivy}\left(\int_{\mathbb{R}}\widehat{f}(w)e^{iwe^{-s}v}\phi_{s}(w)dw  \right)dv  \right)e^{iy xe^{-t}}\phi_{t}(y)dy\\
	&=\frac{1}{(2\pi)^{2}}\int_{\mathbb R}\widehat{f}(w)\phi_{s}(w)\int_{\mathbb R}e^{iy xe^{-t}}\phi_{t}(y) \left(\int_{\mathbb R}e^{iv(we^{-s}-y)}dv  \right)dy dw\\
	&=\frac{1}{(2\pi)^{2}}\int_{\mathbb R}\widehat{f}(w)\phi_{s}(w)\int_{\mathbb R}e^{iy xe^{-t}}\phi_{t}(y) 2\pi \delta (we^{-s} -y) dy dw\\
	&=\frac{1}{2\pi}\int_{\mathbb R}\widehat{f}(w)\phi_{s}(w)e^{i e^{-s}wxe^{-t}}\phi_{t}(we^{-s})dw\\
	&=\frac{1}{2\pi}\int_{\mathbb R}\widehat{f}(w)e^{iwx e^{-(t+s)}}\phi_{t+s}(w) dw\\
	&=P_{t+s}(f)(x).
	\end{align*}
	\noindent
	(ii) 	For $f\in \mathcal{S}(\mathbb{R}),$
	\begin{align}
	\nonumber Lf(x)&=\lim_{t\to0^{+}}\frac{1}{t}\bigg(P_{t}f(x) -f(x)   \bigg)\\
	\nonumber &=\lim_{t\to0^{+}}\frac{1}{t} \bigg(\frac{1}{2\pi} \displaystyle\int_{\mathbb{R}}\widehat{f}(y)e^{iyxe^{-t}}\phi_{t}(y)dy- \frac{1}{2\pi} \displaystyle\int_{\mathbb{R}}\widehat{f}(y)e^{iyx}dy\bigg)\\
	 &=\frac{1}{2\pi}\int_{\mathbb R}\widehat{f}(y)e^{iy x} \lim_{t\to0^{+}}\frac{1}{t}\bigg(e^{iy x(e^{-t}-1)}\phi_{t}(y)-1\bigg)dy, \label{bgnlb0}
	\end{align}
	\noindent
	 by dominated convergence theorem. Now by L'Hospital rule and \eqref{PP2:e4}, we get
	\begin{align}
	\nonumber	\lim_{t\to0^{+}}\frac{1}{t}\big(e^{iy x(e^{-t}-1)}\phi_{t}(y)-1\big)&=-iyx+\lim_{t\to0^{+}} \phi_{t}^{\prime}(y)\\
	\nonumber&=-iyx+\frac{y \phi_{bg}^{\prime}(y)}{\phi_{bg}(y)}\\
		&=-iyx+iy\displaystyle\int_{\mathbb{R}}e^{iy u }u\nu_{bg}(du). \label{bgnlb1}
	\end{align}
\noindent
Using \eqref{bgnlb1} in \eqref{bgnlb0}, we get
\begin{align}
\nonumber	Lf(x)& = \frac{1}{2\pi}\int_{\mathbb R}\widehat{f}(y)e^{iy x}(iy)\left(-x+ \displaystyle\int_{\mathbb{R}}e^{iy u }u\nu_{bg}(du) \right )dy\\
	\nonumber&=-xf^{\prime}(x)+\int_{\mathbb{R}}f^{\prime}(x+u)u\nu_{bg}(du),
	\end{align}
	\noindent
	where the last equality follows by inverse FT.	
\end{proof}
\noindent
Next, we provide the solution of our Stein equation \eqref{PP2:e15}.  
\begin{thm}\label{thmsol}
	Let $X \sim BG(\alpha_1,p_1,\alpha_2,p_2)$ and $h\in\mathcal{W}_{r}$, defined in \eqref{fs1}. Then the function $f_{h}:\mathbb{R}\to \mathbb{R}$ defined by 
	\begin{equation}\label{PP2:SolSe}
	f_{h}(x):=-\displaystyle
	\int_{0}^{\infty}\frac{d}{dx}P_{t}h(x)dt,
	\end{equation}
	solves \eqref{PP2:e15}.	
\end{thm}

\begin{proof}
	Let

	 $$g_{h}(x)=-\displaystyle\int_{0}^{\infty}\bigg(P_{t}(h)(x)-\mathbb{E}h(X)  \bigg)dt .$$

	\noindent
	Then $g_{h}^{\prime}(x)=f_{h}(x).$ Also from \eqref{e7}, we get
	\begin{align}
	\nonumber\mathcal{A}f_{h}(x)&=-xf_{h}(x)+\int_{\mathbb{R}} f_{h}(x+u)u\nu_{bg} (du)\\
	\nonumber&=Lg_{h}(x) \\
	\nonumber&=-\displaystyle\int_{0}^{\infty}LP_{t}(h)(x)dt\\
	\nonumber&=-\displaystyle\int_{0}^{\infty}\frac{d}{dt}P_{t}h(x)dt\text{ (see \cite[p.68]{nourdin})}\\
	\nonumber&=P_{0}h(x)-P_{\infty}h(x)\\
	\nonumber&=h(x)-\mathbb{E}h(X)~(\text{by Proposition \ref{PP2:proSem}}).
	\end{align}
	\noindent
	Hence, $f_{h}$ is the solution to \eqref{PP2:e15}. 	
\end{proof}

\noindent
Next, we obtain some regularity of $f_{h}$. 
\begin{lem}\label{th3}
	For $h\in\mathcal{W}_{k+1}$, let $f_{h}$ be defined in \eqref{PP2:SolSe}. Then
	\begin{align}\label{PP2:pr1}
	\|f_{h}^{(k)}\|\leq \frac{1}{k+1} ,~k\geq 0,
	\end{align}	
	where $f^{(k)}$, $k\geq 1$, denotes the $k$th derivative of $f$ with $f^{(0)}=f$. 
\end{lem}
\begin{proof}
	For $h\in \mathcal{W}_{k+1}$,
	\begin{align*}
	\|f_h\|&=\sup_{x\in\mathbb{R}} \left|-\displaystyle
	\int_{0}^{\infty}\frac{d}{dx}P_{t}h(x)dt\right|\\
	&=\sup_{x\in\mathbb{R}} \left| -\displaystyle
	\int_{0}^{\infty}e^{-t}\int_{\mathbb{R}}h^{(1)}(xe^{-t}+y)F_{X_{(t)}}(dy)dt \right|\\
	&\leq \|h^{(1)}\| \left|\int_{0}^{\infty}e^{-t}dt \right|\\	
	&= \|h^{(1)}\|\\
	&\leq 1.
	\end{align*}
	\noindent
	Since $f_{h}$ is $k$-times differentiable, we have 
	\begin{align*}
	\|f_h^{(1)}\|&=\sup_{x\in\mathbb{R}} \left| -\displaystyle
	\int_{0}^{\infty}e^{-2t}\int_{\mathbb{R}}h^{(2)}(xe^{-t}+y)F_{X_{(t)}}(dy)dt \right|\\
	&\leq \|h^{(2)}\| \left|\int_{0}^{\infty}e^{-2t}dt \right|\\	
	&= \frac{\|h^{(2)}\|}{2}\\
	&\leq \frac{1}{2}.
	\end{align*}
	\noindent
	 Indeed, it follows by induction that
	$$\|f_{h}^{(k)}\|\leq \frac{1}{k+1},~k\geq 1.$$
	\noindent
	This proves the result.
\end{proof}

\section{Bilateral gamma approximation on Wiener space}\label{BSBGD}
\noindent
In this section, we obtain error bounds for $BG$ approximation of functionals of an isonormal Gaussian process $\{Y(h): h\in \mathcal{H} \}$, for the smooth Wasserstein distance $d_r$. The error bounds for $VG$, $SVG$, Laplace, gamma and normal approximations follow as corollaries. Moreover, we give some numerical comparisons of the constants involved in the error bounds between our results and the existing results given by Gaunt \cite{kk2}. First, we need the following integration by parts formulas. Similar type of formulas can also be found in the proof of Theorem 4.1 of \cite{key2}. 

\begin{lem}\label{ibp}
	Let $G\in \mathbb{D}^{2,4}$ and $f:\mathbb{{R}} \to \mathbb{{R}}$ be a twice differentiable function having bounded first and second derivatives. Then,
	\begin{align}
	(i)~\mathbb{E} \left(Gf^{\prime}(G)   \right)=\mathbb{E}(G)\mathbb{E}(f^{\prime}(G))+\mathbb{E}\left(f^{\prime \prime} (G) \Gamma_{2}(G) \right), \text{	and}
	\end{align}
	\begin{align}\label{msidentity1}
	(ii)~\mathbb{E} \left(Gf(G)   \right)=\mathbb{E}(G)\mathbb{E}(f(G))+\mathbb{E}(f^{\prime}(G))\mathbb{E}(\Gamma_{2}(G))+\mathbb{E}\left(f^{\prime\prime}(G) \Gamma_{3}(G) \right),
	\end{align}
	
	where $\Gamma_{j},~j\geq 1,$ is defined as in \eqref{gamop}.
\end{lem}
\begin{proof}
	(i) Recall from the integration by parts formula given in \eqref{ibp00}
	\begin{align}\label{ibp0}
	\mathbb{E}\left(f^{\prime} (G) \Gamma_{2}(G) \right)=\mathbb{E} \left(Gf(G)   \right)-\mathbb{E}(G)\mathbb{E}(f(G)).
	\end{align}
	\noindent
	Replacing $f$ by $f^{\prime}$ in \eqref{ibp0} and rearranging the terms, we get
	\begin{align}
	\mathbb{E} \left(Gf^{\prime}(G)   \right)=\mathbb{E}(G)\mathbb{E}(f^{\prime}(G))+\mathbb{E}\left(f^{\prime \prime} (G) \Gamma_{2}(G) \right).
	\end{align}
	\noindent
	(ii) Using the fact $LL^{-1}G=G-\mathbb{E}(G) $, we observe that
	\begin{align}
	\nonumber\mathbb{E}\bigg(\bigg(\Gamma_{2}(G)-\mathbb{E}(\Gamma_{2}(G))\bigg)f^{\prime}(G)  \bigg)&=\mathbb{E}\bigg(LL^{-1}\Gamma_{2}(G) f^{\prime}(G)   \bigg)\\
	\nonumber&=\mathbb{E} \bigg\{\delta \left(-DL^{-1}\Gamma_{2}(G)   \right) f^{\prime} (G)   \bigg\} \\
	\nonumber&\quad\quad\quad\quad\quad\quad \text{ (since $\delta(DG)=-LG$)}\\
	\nonumber&=\mathbb{E} \bigg( \langle Df^{\prime}(G), -DL^{-1}\Gamma_{2}(G) \rangle_{\mathcal{H}}    \bigg) \text{ (by \eqref{dualfo})}\\
	\nonumber&=\mathbb{E}\bigg(f^{\prime\prime}(G)\langle DG, -DL^{-1}\Gamma_{2}(G) \rangle_{\mathcal{H}}    \bigg) \text{ (by chain rule)}\\
	&=\mathbb{E}\bigg(f^{\prime\prime}(G) \Gamma_{3}(G)   \bigg),\label{ibp1}
	\end{align}
	which leads to
	\begin{align}
	\mathbb{E}\bigg(f^{\prime}(G)\Gamma_{2}(G)  \bigg)=\mathbb{E}\bigg(f^{\prime}(G)\bigg)\mathbb{E}\bigg(\Gamma_{2}(G)\bigg)+\mathbb{E}\bigg(f^{\prime\prime}(G) \Gamma_{3}(G)   \bigg)\label{ibp2}
	\end{align}
	\noindent
	Putting \eqref{ibp0} in \eqref{ibp2}, we arrive at the identity in \eqref{msidentity1}.
\end{proof}

\subsection{Bounds in terms of $\Gamma$-operators} Here, we establish our bounds, for the $d_3$ distance, in terms of $\Gamma$-operators. Our approach is new and exploits Stein equation based on integral operator, rather than the one based on differential forms.
\begin{thm}\label{bgapp:th1}
	Let $G \in \mathbb{D}^{2,4}$ and $X\sim BG(\alpha_1,p_1,\alpha_2,p_2)$, and $\Gamma_{j}$, $j\geq 1$, denote the $\Gamma$-operators defined in \eqref{gamop}.  Also, let $\alpha_1\alpha_2>1+|\alpha_1-\alpha_2|$,
	\begin{align}\label{alpi}
		\alpha_{12}=\frac{\alpha_1 \alpha_2}{\alpha_1\alpha_2-(1+|\alpha_1-\alpha_2|)} \text{ and } \alpha_{13}=\frac{1}{\alpha_1}-\frac{1}{\alpha_2}.
	\end{align}
	Then,
	\begin{align}
		\nonumber d_{{3}}(G,X) \leq&~ \frac{1}{3}\alpha_{12}\mathbb{E} \left|\frac{1}{\alpha_1 \alpha_2}G+ \alpha_{13}\Gamma_{2}(G)-\Gamma_{3}(G)   \right|\\
		&\quad\quad+\frac{1}{2} \alpha_{12} \left|\frac{p_1+p_2}{\alpha_1 \alpha_2}+\alpha_{13}\mathbb{E}(G)- \mathbb{E}\left(\Gamma_{2}(G) \right)  \right|+\alpha_{12} \bigg| \mathbb{E} (X)-\mathbb{E}(G)\bigg|.\label{sm6}
	\end{align}
\end{thm}
\begin{proof}
	Recall the L\'evy measure $\nu_{bg}$ of $BG(\alpha_1,p_1,\alpha_2,p_2)$, given in \eqref{bglevymeasure}, is
	$$\nu_{bg}(du)=\left(\frac{p_1  }{u}e^{-\alpha_1u}\mathbf{1}_{(0,\infty)}(u)-\frac{p_2  }{u}e^{-\alpha_{2}|u|}\mathbf{1}_{(-\infty,0)}(u)\right)du.$$
	
	\noindent
	Also,
	$$\mathbb{E}(X)=\left(\frac{p_1}{\alpha_1}-\frac{p_2}{\alpha_2}   \right).$$
	 \noindent
	Let $(f,h)$ be the solution pair of \eqref{PP2:e15}. Then replacing $x$ by $G \in \mathbb{D}^{2,4}$ in the Stein equation and taking expectation, we get
	\begin{align}
	\nonumber	\mathbb{E}h(G)-\mathbb{E}h(X)=&\mathbb{E}\left(-Gf(G)+\displaystyle\int_{\mathbb{R}}f(G+u)u\nu_{bg}(du)  \right)\\
		\nonumber=&\mathbb{E} \bigg(-Gf(G)+p_1 \displaystyle\int_{0}^{\infty}f(G+u)e^{-\alpha_1 u}du\\
	\nonumber	&\quad\quad\quad\quad\quad\quad -  p_2 \displaystyle\int_{0}^{\infty}f(G-u)e^{-\alpha_2 u}du \bigg)\\
	\nonumber&\quad\quad\quad\quad\quad\quad\text{(splitting the L\'evy measure $\nu_{bg}$)}\\
		\nonumber=& \mathbb{E}\bigg( \left(  \left(\frac{p_1}{\alpha_1}-\frac{p_2}{\alpha_2}   \right) -G \right)f(G)\bigg)+\frac{p_1}{\alpha_1}\mathbb{E} \left(\displaystyle\int_{0}^{\infty}f^{\prime}(G+u)e^{-\alpha_1 u} du \right)\\
		\label{ibpf1}&\quad\quad+\frac{p_2}{\alpha_2}\mathbb{E} \left(\displaystyle\int_{0}^{\infty}f^{\prime}(G-u)e^{-\alpha_2 u} du \right), 
		\end{align}
		\noindent
		where the last equality follows by the integration by parts formula. Note that
		\begin{align}\label{ibpf2}
			\mathbb{E}\bigg(p_1 \int_{0}^{\infty}f^{\prime}(G+u)e^{-\alpha_1u}du-	p_2\int_{0}^{\infty}f^{\prime}(G-u)e^{-\alpha_2u}du\bigg)=\mathbb{E}\bigg(\displaystyle\int_{\mathbb R}uf^{\prime}(G+u)\nu_{bg}(du)\bigg).
		\end{align}
		\noindent
		Using \eqref{ibpf2} in \eqref{ibpf1}, we get
		\begin{align}
		\nonumber\mathbb{E}h(G)-\mathbb{E}h(X)=&\mathbb{E} \bigg(\left(  \mathbb{E}X -G \right)f(G)\bigg)+\left(\frac{1}{\alpha_1}-\frac{1}{\alpha_2} \right)\mathbb{E} \bigg( p_1 \int_{0}^{\infty}f^{\prime}(G+u)e^{-\alpha_1u}du\\
		\nonumber&\quad-	p_2\int_{0}^{\infty}f^{\prime}(G-u)e^{-\alpha_2u}du \bigg)+\frac{p_1}{\alpha_2}\mathbb{E} \left(\displaystyle\int_{0}^{\infty}f^{\prime}(G+u)e^{-\alpha_1 u} du \right)\\
		\nonumber&\quad\quad\quad\quad\quad\quad+\frac{p_2}{\alpha_1}\mathbb{E} \left(\displaystyle\int_{0}^{\infty}f^{\prime}(G-u)e^{-\alpha_2 u} du \right)\\
		\nonumber=&\mathbb{E} \bigg(\left(  \mathbb{E}X -G \right)f(G)\bigg)+\alpha_{13}\mathbb{E} \left( \int_{\mathbb{R}} f^{\prime}(G+u) u\nu_{bg}(du) \right)\\
		\nonumber&\quad\quad\quad+\frac{p_1}{\alpha_2}\mathbb{E} \left(\displaystyle\int_{0}^{\infty}f^{\prime}(G+u)e^{-\alpha_1 u} du \right)\\
			\label{ibpf3}&\quad\quad\quad\quad\quad\quad+\frac{p_2}{\alpha_1}\mathbb{E} \left(\displaystyle\int_{0}^{\infty}f^{\prime}(G-u)e^{-\alpha_2 u} du \right).
			\end{align}
		\noindent
		Applying integration by parts formula to the last two terms of \eqref{ibpf3}, we get	
			\begin{align}
	 \nonumber\mathbb{E}h(G)-\mathbb{E}h(X)=&\mathbb{E} \bigg(\left(  \mathbb{E}X -G \right)f(G)\bigg)+\alpha_{13}\mathbb{E} \left( \int_{\mathbb{R}} f^{\prime}(G+u) u\nu_{bg}(du) \right)\\
	 \nonumber&+\bigg(\frac{p_1+p_2}{\alpha_1 \alpha_2}\bigg)\mathbb{E}\left(f^{\prime}(G) \right)+	\frac{1}{\alpha_1 \alpha_2}\mathbb{E}\bigg(p_1 \int_{0}^{\infty}f^{\prime\prime}(G+u)e^{-\alpha_1u}du\\
	 \nonumber&-p_2\int_{0}^{\infty}f^{\prime\prime}(G-u)e^{-\alpha_2u}du\bigg)\\
	 =&\mathbb{E} \bigg(\left(  \mathbb{E}X -G \right)f(G)\bigg)+\alpha_{13}\mathbb{E} \left( \int_{\mathbb{R}} f^{\prime}(G+u) u\nu_{bg}(du) \right)\\
	 \label{ibpf4}&+\bigg(\frac{p_1+p_2}{\alpha_1 \alpha_2}\bigg)\mathbb{E}\left(f^{\prime}(G) \right)+\frac{1}{\alpha_1\alpha_2}\mathbb{E} \left( \int_{\mathbb{R}} f^{\prime\prime}(G+u) u\nu_{bg}(du) \right),
	 		\end{align}
\noindent
where the last integral follows by using \eqref{ibpf2} (replacing $f^\prime$ by $f^{\prime\prime}$). Also, \eqref{ibpf4} can be written as

\begin{align}
		\nonumber\mathbb{E}h(G)-\mathbb{E}h(X) =&\mathbb{E}\bigg(\frac{1}{\alpha_1 \alpha_2}Gf^{\prime\prime}(G)+ \left(\frac{p_1+p_2}{\alpha_1\alpha_2}+ \alpha_{13}G   \right)f^{\prime}(G)\\
		\nonumber&\quad +\left(  \mathbb{E}X -G \right)f(G)  \bigg)\\
		\nonumber&\quad\quad+\alpha_{13}\mathbb{E} \left(-Gf^{\prime}(G)+ \int_{\mathbb{R}} f^{\prime}(G+u) u\nu_{bg}(du) \right)\\
		&\quad\quad\quad+ \frac{1}{\alpha_1\alpha_2} \mathbb{E} \left(-Gf^{\prime\prime}(G)+ \int_{\mathbb{R}} f^{\prime\prime}(G+u) u\nu_{bg}(du) \right).\label{tt0}
	\end{align}
	\noindent
	Let $h_1,h_2 \in \mathcal{W}_3$ be two associated test functions with $f^\prime$ and $f^{\prime\prime}$ satisfy the Stein equation \eqref{PP2:e15}. Then from \eqref{tt0}, we get
	\begin{align}
	\nonumber\mathbb{E}h(G)-\mathbb{E}h(X) =&\mathbb{E}\bigg(\frac{1}{\alpha_1 \alpha_2}Gf^{\prime\prime}(G)+ \left(\frac{p_1+p_2}{\alpha_1\alpha_2}+ \alpha_{13}G   \right)f^{\prime}(G)\\
	\nonumber&\quad +\left(  \mathbb{E}X -G \right)f(G)  \bigg)+\alpha_{13}\bigg(\mathbb{E}h_1(G)-\mathbb{E}h_1(X)\bigg)\\
	&\quad\quad\quad+ \frac{1}{\alpha_1\alpha_2} \bigg(\mathbb{E}h_2(G)-\mathbb{E}h_2(X)\bigg).\label{ibpf5}
	\end{align}
	\noindent
	For any $h \in \mathcal{W}_3$,
	\begin{align}
	\nonumber\bigg|\mathbb{E}h(G)-\mathbb{E}h(X)\bigg| \leq&\bigg|\mathbb{E}\bigg(\frac{1}{\alpha_1 \alpha_2}Gf^{\prime\prime}(G)+ \left(\frac{p_1+p_2}{\alpha_1\alpha_2}+ \alpha_{13}G   \right)f^{\prime}(G)\\
	\nonumber&\quad +\left(  \mathbb{E}X -G \right)f(G)  \bigg)\bigg|+\bigg|\alpha_{13}\bigg(\mathbb{E}h_1(G)-\mathbb{E}h_1(X)\bigg)\bigg|\\
	&\quad\quad\quad+ \frac{1}{\alpha_1\alpha_2} \bigg|\bigg(\mathbb{E}h_2(G)-\mathbb{E}h_2(X)\bigg)\bigg|.\label{ibpf6}
	\end{align}
	\noindent
	Taking supremum over the functions $h \in \mathcal{W}_3,$ we set
	\begin{align}
	\nonumber d_3(G,X) \leq&\bigg|\mathbb{E}\bigg(\frac{1}{\alpha_1 \alpha_2}Gf^{\prime\prime}(G)+ \left(\frac{p_1+p_2}{\alpha_1\alpha_2}+ \alpha_{13}G   \right)f^{\prime}(G)\\
	&\quad +\left(  \mathbb{E}X -G \right)f(G)  \bigg)\bigg|+|\alpha_{13}|d_3(G,X)+ \frac{1}{\alpha_1\alpha_2}d_3(G,X).\label{ibpf7}
	\end{align}
	\noindent
	That is,
	\begin{align}
	\nonumber&\left(1 - |\alpha_{13}|-\frac{1}{\alpha_1 \alpha_2}\right)d_{3}(G,X)\\	
	 &\quad\quad \leq \bigg|  \mathbb{E}\bigg(\frac{1}{\alpha_1 \alpha_2}Gf^{\prime\prime}(G)+ \left(\frac{p_1+p_2}{\alpha_1\alpha_2}+ \alpha_{13}G   \right)f^{\prime}(G)+\left(  \mathbb{E}X -G \right)f(G)  \bigg)\bigg|.\label{tt1}
	\end{align}
	\noindent
	 Using Lemma \ref{ibp} for $\mathbb{E}(Gf(G))$ and $\mathbb{E}(Gf^{\prime}(G))$, we get after rearranging the terms	
	\begin{align}
		\nonumber&\left(1 - |\alpha_{13}|-\frac{1}{\alpha_1 \alpha_2}\right)d_{3}(G,X)\\	
	\nonumber&\quad\quad \leq \bigg|  \mathbb{E}\bigg(\frac{1}{\alpha_1 \alpha_2}Gf^{\prime\prime}(G)+ \left(\frac{p_1+p_2}{\alpha_1\alpha_2}+ \alpha_{13}G   \right)f^{\prime}(G)+\left(  \mathbb{E}X -G \right)f(G)  \bigg)\bigg|\\
		\nonumber&\quad\quad=\bigg|\mathbb{E}\bigg(f^{\prime\prime}(G) \left(\frac{1}{\alpha_1\alpha_2}G +\alpha_{13}\Gamma_{2}(G)-\Gamma_{3}(G)  \right)\\
		\nonumber&\quad\quad\quad\quad+f^{\prime}(G) \left( \frac{p_1+p_2}{\alpha_1 \alpha_2}+\alpha_{13}\mathbb{E}(G)-  \mathbb{E}(\Gamma_{2}(G)) \right) +f(G)\bigg(\mathbb{E}(X)-\mathbb{E}(G)\bigg)    \bigg)     \bigg| \\
		\nonumber&\quad\quad\leq \|f^{\prime\prime}\|\mathbb{E}\bigg| \frac{1}{\alpha_1\alpha_2}G +\alpha_{13}\Gamma_{2}(G)-\Gamma_{3}(G)  \bigg|\\
		&\quad\quad\quad\quad+\|f^{\prime}\| \bigg|\frac{p_1+p_2}{\alpha_1\alpha_2}+\alpha_{13}\mathbb{E}(G) - \mathbb{E}(\Gamma_{2}(G))   \bigg|+\|f\| \bigg| \mathbb{E} (X)-\mathbb{E}(G)\bigg|.\label{tt4}
	\end{align}
	\noindent
	Using Lemma \ref{th3} in \eqref{tt4}, we get	
	\begin{align}
	\nonumber 	d_{3}(G,X) &\leq  \frac{1}{3}\alpha_{12}\mathbb{E}\bigg| \frac{1}{\alpha_1\alpha_2}G +\alpha_{13}\Gamma_{2}(G)-\Gamma_{3}(G)  \bigg|\\
\nonumber	&\quad\quad+\frac{1}{2}\alpha_{12} \bigg|\frac{p_1+p_2}{\alpha_1\alpha_2}+\alpha_{13}\mathbb{E}(G) - \mathbb{E}(\Gamma_{2}(G))   \bigg|+ \alpha_{12}\bigg| \mathbb{E} (X)-\mathbb{E}(G)\bigg|, 
	\end{align}
	\noindent
	  where $\alpha_{12}$ is given in \eqref{alpi}. This proves the result.
\end{proof}
\begin{rem}
	(i) Observe that when $\alpha_1,\alpha_2>1$, the condition $\alpha_1 \alpha_2 > 1+ |\alpha_1- \alpha_2|$ is satisfied, and in that case $0 \leq |\alpha_{13}| \leq 1.$ It will be interesting to obtain bounds for the case $\alpha_1<1$ or $\alpha_2<1$ or both are less than unity.
	
	\vskip 1ex
	\item [(ii)] The first term on the right hand side of \eqref{sm6} provides a constant times of the $L^1$-distance between $G^{*}:=\frac{1}{\alpha_1 \alpha_2}G+\alpha_{13}\Gamma_{2}(G)$ and $\Gamma_{3}(G)$, which plays a crucial role to present an upper bound of $d_{3}(G,X)$ in terms of cumulants (see Theorem \ref{cumulcor0}).
	
	\vskip 1ex
	\item[(iii)] From \eqref{sm6}, if $\mathbb{E}(G) = \mathbb{E}(X)$, then
	\begin{align}
			 \nonumber d_{{3}}(G,X) \leq&~ \frac{1}{3} \alpha_{12} \mathbb{E} \left|G^*-\Gamma_{3}(G)   \right|+\frac{1}{2} \alpha_{12} \mathbb{E} \left|\left(\frac{p_1}{\alpha_1^{2} } +\frac{p_2}{\alpha_2^{2} }\right)- \mathbb{E}\left(\Gamma_{2}(G) \right)  \right|\\ =&~\frac{1}{3} \alpha_{12} \mathbb{E} \left|G^*-\Gamma_{3}(G)   \right|+\frac{1}{2} \alpha_{12} \mathbb{E} \left|Var(X)- Var(G) \right|.\label{re0}
	\end{align}
	\noindent
	 So, when $\mathbb{E}(G)=\mathbb{E}(X)$ and $Var(G)=Var(X),$ then from \eqref{re0},
	$$d_{{3}}(G,X) \leq~ \frac{1}{3} \alpha_{12} \mathbb{E} \left|G^*-\Gamma_{3}(G)   \right|.$$
	
\end{rem}
\noindent
The following corollary immediately follows for $VG(\alpha_1,\alpha_2,p)$, $SVG(\alpha,p)$ and $La(\alpha)$ distributions.
\begin{cor}\label{vgapproxcor}
	Let $G\in \mathbb{D}^{2,4}$.
	\noindent
		\item [(i)] Let $X_{v}\sim VG(\alpha_1,\alpha_2,p)$, where $\alpha_1\alpha_2>(1+|\alpha_1-\alpha_2|)$, and $\alpha_{12}$ and $\alpha_{13}$ be defined as in \eqref{alpi}. Then,
		\begin{align}
		\nonumber d_{{3}}(G,X_v) \leq&~ \frac{1}{3} \alpha_{12} \mathbb{E} \left|\frac{1}{\alpha_1 \alpha_2}G+ \alpha_{13}\Gamma_{2}(G)-\Gamma_{3}(G)   \right|\\
		&+\frac{1}{2} \alpha_{12}  \left|\frac{2p}{\alpha_1 \alpha_2}+\alpha_{13}\mathbb{E}(G)- \mathbb{E}\left(\Gamma_{2}(G) \right)  \right|+\alpha_{12} | \mathbb{E} (X)-\mathbb{E}(G)|.\label{msb2}
		\end{align}
		\item [(ii)] Let $X_s\sim SVG(\alpha,p)$, where $\alpha>1$. Then,
		\begin{align}
		\nonumber d_{{3}}(G,X_s)& \leq \frac{\alpha^{2}}{3(\alpha^{2}-1)}\mathbb{E}\left| \frac{1}{\alpha^{2}}G-\Gamma_{3}(G)  \right|\\
		 &\quad\quad+\frac{\alpha^{2}}{2(\alpha^{2}-1)}\left|\frac{2p}{\alpha^{2}}-\mathbb{E}\left(\Gamma_{2}(G) \right)  \right|+\frac{\alpha^{2}}{(\alpha^{2}-1)}| \mathbb{E} (X_s)-\mathbb{E}(G)|. \label{PP2:smt1}
		\end{align}
		
		\item [(iii)] Let $X_l\sim La(\alpha)$. Then, for $\alpha>1$,
		\begin{align}
			\nonumber d_{{3}}(G,X_l) & \leq \frac{\alpha^{2}}{3(\alpha^{2}-1)}\mathbb{E}\left| \frac{1}{\alpha^{2}}G-\Gamma_{3}(G)  \right|\\
			 &\quad\quad+\frac{\alpha^{2}}{2(\alpha^{2}-1)}\left|\frac{2}{\alpha^{2}}-\mathbb{E}\left(\Gamma_{2}(G) \right)  \right|+\frac{\alpha^{2}}{(\alpha^{2}-1)}| \mathbb{E} (X_l)-\mathbb{E}(G)|.\label{msb3}
		\end{align}
\end{cor}

\begin{rem}
	(i) Eichelsbacher and Th$\ddot{\text{a}}$le \cite[Theorem 4.1]{key2} derived an upper bound in terms of $\Gamma$-operators, under the Wasserstein distance $d_W$, for centered $VG$ approximation of functionals of an isonormal Gaussian process. However, their bound contains two indeterminate positive constants. 
	
	\vskip 1ex
	\item [(ii)] Recently, Gaunt \cite[Theorem 4.1]{gauntnew} obtained two explicit constants for the $VG$ approximation, but they are too complicated to compute. For the $SVG$ case with $\mathbb{E}(G)=\mathbb{E}(X_s)=0$, Gaunt \cite[Theorem 5.3]{kk2} obtained the following bound:
	\begin{align}
	 d_{W}(G, X_s)& \leq b_1\mathbb{E}\left| \frac{1}{\alpha^{2}}G-\Gamma_{3}(G)  \right|+b_2\left|\frac{2p}{\alpha^{2}}-\mathbb{E}\left(\Gamma_{2}(G) \right)  \right| ,\label{svgsm2}
   \end{align}     
\noindent
where $b_1=\frac{9\alpha^{2}}{(2p+1)}$ and $b_2=\frac{9 \alpha}{2}\bigg(\frac{1}{2p+1}+\frac{\pi}{2} \frac{\Gamma (p+1/2)}{\Gamma(p+1)}\bigg).$ For the case $\mathbb{E}(G)=\mathbb{E}(X_s)=0$, our bound for the $d_3$ metric is
\begin{align}
d_{3}(G, X_s)& \leq c_1\mathbb{E}\left| \frac{1}{\alpha^{2}}G-\Gamma_{3}(G)  \right|+c_2\left|\frac{2p}{\alpha^{2}}-\mathbb{E}\left(\Gamma_{2}(G) \right)  \right| .\label{svgmsb1}
\end{align}
\noindent
where $c_1=\frac{\alpha^{2}}{3(\alpha^{2}-1)}$ and $c_2=\frac{\alpha^{2}}{2(\alpha^{2}-1)}.$ Note that only constants $b_i$ and $c_i$, $i=1,2$, are different. Also, $c_1<b_1$ if

\begin{align}\label{svgmsb2}
	\alpha > \bigg(1+ \frac{2p+1}{27}\bigg)^{\frac{1}{2}}.
\end{align}
\noindent
Also, when, \eqref{svgmsb2} holds then $c_2 <b_2$ if
$$\frac{1}{2p+1}+\frac{\pi}{2} \frac{\Gamma (p+1/2)}{\Gamma(p+1)}> \frac{\alpha}{9(\alpha^2 -1)},$$
\noindent
which will be true if
\begin{align}\label{svgmsb3}
\frac{1}{2p+1}+\frac{\pi}{2} \frac{\Gamma (p+1/2)}{\Gamma(p+1)}>\frac{\alpha}{9} \frac{27}{(2p+1)},
\end{align}
\noindent
since, from \eqref{svgmsb2}, $$\frac{1}{\alpha^2 -1} < \frac{27}{2p+1}.$$
\noindent
Note that \eqref{svgmsb3} holds when
\begin{align}\label{svgmsb4}
\alpha< \frac{1}{3}+\frac{\pi}{6}\frac{(2p+1) \Gamma (p+1/2)}{\Gamma(p+1)}.
\end{align}
\noindent
Hence, when $c_1<b_1$ and $c_2<b_2$ and
$$\bigg(1+ \frac{2p+1}{27}\bigg)^{\frac{1}{2}}<\alpha<\frac{1}{3}+\frac{\pi}{6}\frac{(2p+1) \Gamma (p+1/2)}{\Gamma(p+1)},$$
\noindent
 the bound for $d_3$ given in \eqref{svgmsb1} is sharper than the one given in \eqref{svgsm2}.
\end{rem}

\noindent
Our next result deals with the normal and gamma distributions.
\begin{cor}\label{norgam0}
	Let $G\sim \mathbb{D}^{2,4},$ and $\alpha>1$.
	\begin{enumerate}
		\item [(i)] Let $Z_\alpha \sim \mathcal{N}(0,\alpha^2)$. Then
		\begin{align}\label{norgam1}
			d_3(G,Z_\alpha) \leq \frac{1}{3}\mathbb{E}(|\Gamma_{3}(G)|)+\frac{1}{2}\bigg|\alpha^2- \mathbb{E}(\Gamma_{2}(G))\bigg|+|\mathbb{E}(Z_\alpha)-\mathbb{E}(G)|.
		\end{align}
		\item [(ii)] Let $X_g \sim Ga(\alpha,p)$. Then
		\begin{align}
		\nonumber d_3(G,X_g) &\leq	\frac{\alpha}{3(\alpha-1)}\mathbb{E}\bigg|\frac{1}{\alpha}\Gamma_{2}(G)-\Gamma_{3}(G)  \bigg|\\
		\nonumber&\quad\quad+\frac{\alpha}{2(\alpha-1)}\bigg|\frac{1}{\alpha}\mathbb{E}(G)-\mathbb{E}(\Gamma_{2}(G)) \bigg|\\
		&\quad\quad\quad\quad+\frac{\alpha}{(\alpha-1)}\bigg|\mathbb{E}(X_g)-\mathbb{E}(G) \bigg|\label{norgam2}
		\end{align}
	\end{enumerate}
\end{cor}
\begin{proof}
	(i) Let $\alpha_p= \frac{\sqrt{2p}}{\alpha} ~(\alpha>1)$. Also, let $X_p \sim SVG(\alpha_p,p),$ and $Z_\alpha \sim \mathcal{N}(0,\alpha^2)$. Then it can be checked using the cf that $X_p \overset{L}{\to} Z_\alpha$, as $p\to \infty$. Also, it follows (see Theorem 7.12 of \cite{Villani}) that

		$$\lim_{p\to \infty}d_3(G,X_p)=d_3(G,Z_\alpha).$$
	\noindent
	Now replacing $X_s$ by $X_p$ and $\alpha$ by $\alpha_p$  in \eqref{PP2:smt1}, and taking $p\to \infty$, \eqref{norgam1} follows.
	
\vskip 1ex
	\noindent
	(ii) Since the $BG(\alpha_1,p_1,\alpha_2,p_2)$ tends to $Ga(\alpha_1,p_1)$ distribution, as $p_2 \to \infty,$ using the arguments as in (i) above, writing  $\alpha_1=\alpha$ and $p_1=p$, the result follows.
\end{proof}

\begin{rem}
(i) Eichelsbacher and Th$\ddot{\text{a}}$le \cite[Corollary 4.4]{key2} derived the following bounds for $G\in \mathbb{D}^{2,4}$ with $\mathbb{E}(G)=0$
	\begin{align}\label{Ei0}
		d_W(G,Z_\alpha) \leq c_3 \mathbb{E}\left(|\Gamma_3(G)|\right)+c_4\bigg|\alpha^2 - \mathbb{E}(\Gamma_2(G))\bigg|,
	\end{align}	
\noindent
and
\begin{align}\label{Ei1}
	d_W(G,X_g) \leq c_5\mathbb{E} \left|\frac{1}{\alpha}\Gamma_2(G)-\Gamma_{3}(G) \right|+c_6\left|\frac{p}{\alpha^2}-\mathbb{E}(\Gamma_2(G)) \right|,
\end{align}

\noindent
where the constants $c_3$ - $c_6$ are mentioned to be positive, but explicit forms are not given.

\vskip 1ex	
\item[(ii)] From \eqref{norgam1} with $\mathbb{E}(G)=\mathbb{E}(Z_\alpha)=0$, we get the following bound for the metric $d_3$ as
\begin{align}\label{norgam3}
	d_3(G,Z_\alpha) \leq \frac{1}{3}\mathbb{E}|\Gamma_{3}(G)|+\frac{1}{2}\left|\alpha^2- \mathbb{E}(\Gamma_{2}(G))\right|.
\end{align}
\noindent
From \eqref{norgam2}, if $\mathbb{E}(G)= \mathbb{E}(X_g)=\frac{p}{\alpha}$, then

\begin{align}\label{norgam4}
 d_3(G,X_g) &\leq	\frac{\alpha}{3(\alpha-1)}\mathbb{E}\bigg|\frac{1}{\alpha}\Gamma_{2}(G)-\Gamma_{3}(G)  \bigg|+\frac{\alpha}{2(\alpha-1)}\bigg|\frac{p}{\alpha^2}-\mathbb{E}(\Gamma_{2}(G)) \bigg|.
\end{align}
\noindent
Observe that our bound given in \eqref{norgam3} and \eqref{norgam4} contains the explicit constants, and this is one of the advantages of our approach.
\end{rem}

\subsection{Bounds in terms of cumulants}
The next result presents upper bound for $BG$ approximation, for $d_3$ in terms of the cumulants, if a sequence of rvs is of the form $G_n=I_2(f_n),~f_n\in \mathcal{H}^{\odot 2},~n\geq 1.$ First we prove a lemma which is helpful for Theorem \ref{cumulcor0}.
\begin{lem}\label{gammopformula}
	Let $G=I_2(f)$ where $f\in \mathcal{H}^{\odot 2}$. Then for $m,p\geq 1,$
	\begin{align}
	(i)~ \mathbb{E} \bigg(G \Gamma_{m} (G)  \bigg)&=\frac{1}{m!}\kappa_{m+1}(G),~\text{and}\\
	(ii)~\mathbb{E} \bigg(\Gamma_{m}(G) \Gamma_{p}(G)   \bigg)&=\frac{\kappa_{m+p} (G)}{(m+p-1)!}+ \frac{\kappa_{m} (G) \kappa_{p} (G)}{ (m-1)! (p-1)!   }.\label{gammopformula2}
	\end{align}
\end{lem} 
\vfill

\begin{proof}
	(i) For $G,H \in \mathbb{D}^{1,2}$, we have from \eqref{intbypfo}
	
	\begin{align}\label{ibpformula0}
	\mathbb{E}(GH)=\mathbb{E}(G)\mathbb{E}(H)+ \mathbb{E}\bigg(\langle DG, -DL^{-1}H \rangle_{\mathcal{H}} \bigg).
	\end{align}	
	\noindent
	It is also known that (see Proposition 2.7.13 of \cite{nourdin}), for a rv $G=I_2(f)$ with $f\in \mathcal{H}^{\odot 2}$, $\mathbb{E}(G)=0$. Let $H=\Gamma_m(G) \in \mathbb{D}^{\infty}, ~m\geq 1$. Then from \eqref{ibpformula0} and \eqref{gamop}, we get
	\begin{align*}
	\mathbb{E}(G \Gamma_m(G))&=\mathbb{E} \bigg(\langle DG,-DL^{-1}\Gamma_{m}(G) \rangle_{\mathcal{H}}  \bigg)\\
	&=\mathbb{E}\bigg( \Gamma_{m+1}(G) \bigg) \\
	&=\frac{\kappa_{m+1} (G)}{m!} \text{ (by \eqref{CuGamma})}.
	\end{align*}
	
	\noindent
	(ii) To prove part (ii) we use the following relation (see Equation (24) of \cite{MSLC0})
	\begin{align}\label{ibpformula2}
	I_2(f \otimes_{1}^{(j)} f )=\frac{1}{2^{j-1}} \bigg(\Gamma_{j}(G)-\mathbb{E}(\Gamma_{j}(G) ) \bigg),~j\geq1.
	\end{align}
	\noindent
	From \eqref{CuGamma} and \eqref{ibpformula2}, we have for $m,p\geq 1$,
	\begin{align*}
	\mathbb{E} \bigg( \Gamma_{m}(G) \Gamma_{p}(G) \bigg)&=\mathbb{E} \bigg[\bigg(\Gamma_{m}(G)-\mathbb{E} (\Gamma_{m}(G))\bigg)\bigg(\Gamma_{p}(G)-\mathbb{E}(\Gamma_{p}(G))  \bigg)   \bigg]\\
	& \quad\quad+\mathbb{E}(\Gamma_m (G)) \mathbb{E}(\Gamma_p(G) )\\
	&=2^{(m+p-2)}\mathbb{E}\bigg(	I_2(f \otimes_{1}^{(m)} f )	I_2(f \otimes_{1}^{(p)} f )  \bigg)+\frac{\kappa_m(G) \kappa_p (G)}{(m-1)!(p-1)!}\\
	&=2^{(m+p-1)} \langle f \otimes_{1}^{(m)}f,f \otimes_{1}^{(p)}f\rangle _{\mathcal{H}^{\odot 2}}+\frac{\kappa_m(G) \kappa_p (G)}{(m-1)!(p-1)!}\\
	& \quad\quad\quad\quad\quad\quad\quad\quad\quad\quad\quad\quad\quad\quad\quad\quad\quad\quad\quad\text{ (using \eqref{prductfor2})}\\
	&=2^{(m+p-1)}\langle f \otimes_{1}^{(m+p-1)}f,f \rangle _{\mathcal{H}^{\odot 2}}+\frac{\kappa_m(G) \kappa_p (G)}{(m-1)!(p-1)!}\\
	& \quad\quad\quad\quad\quad\quad\quad\quad\quad\quad\quad\quad\quad\quad\quad\quad\quad\quad\quad\text{ (using \eqref{contract1})}\\
	&=\frac{\kappa_{m+p}(G)}{(m+p-1)!}+\frac{\kappa_m(G) \kappa_p (G)}{(m-1)!(p-1)!},
	\end{align*}
	\noindent
	where the last equality follows by \eqref{sechaoscum}.
\end{proof}
\begin{thm}\label{cumulcor0}
	Let $X\sim BG(\alpha_1,p_1,\alpha_{2},p_2)$ such that $\alpha_1\alpha_2>(1+|\alpha_1-\alpha_{2}|)$. Consider the sequence $G_n=I_2(f_{n}),~n\geq 1$, such that $f_{n} \in \mathcal{H}^{\odot 2}  $. Then 
	\begin{align}
		\nonumber d_{{3}}(G_{n}, X) &\leq \frac{1}{3} \alpha_{12} \bigg\{\frac{1}{120}\kappa_{6}(G_n)- \frac{1}{12}\alpha_{13}\kappa_{5}(G_n)\\
		\nonumber &\quad+\left(\frac{1}{6}\alpha_{13}^{2} -\frac{1}{3\alpha_1 \alpha_2}  \right)\kappa_{4}(G_n)+\frac{1}{\alpha_1\alpha_2}\alpha_{13}\kappa_{3}(G_n)+\frac{1}{4}\kappa_{3}^{2}(G_n)\\
	\nonumber	&\quad\quad- \alpha_{13}\kappa_{2}(G_n)\kappa_{3}(G_n)+\frac{1}{\alpha_1^{2} \alpha_2^{2}}\kappa_{2}(G_n)+\alpha_{13}^{2}\kappa_{2}^{2}(G_n)    \bigg\}^{1/2}\\
	 &\quad\quad\quad+\frac{1}{2}\alpha_{12} \left|\frac{p_1+p_2}{\alpha_1\alpha_2} -\kappa_{2}(G_n)  \right|+ \alpha_{12}\bigg|\frac{p_1}{\alpha_{1}} -\frac{p_2}{\alpha_2}\bigg|, \label{sm7}
	\end{align}
	where $\alpha_{12}$ and $\alpha_{13}$ are defined in \eqref{alpi}.
\end{thm}
\begin{proof}
	First note that
	\begin{align}\label{sm0}
		\mathbb{E}(G_n)=0;~\mathbb{E}\left( \Gamma_{2}(G_{n}) \right)=\kappa_{2}(G_{n})=Var(G_n),~n\geq 1.
	\end{align}
\noindent
Note that the first term (up to a constant) in \eqref{sm6},
\begin{align}
	\nonumber&\mathbb{E}\bigg| \frac{1}{\alpha_1\alpha_2}G_n +\alpha_{13}\Gamma_{2}(G_n)-\Gamma_{3}(G_n) \bigg|\\
\nonumber	&\quad\quad \leq \left(\mathbb{E}\bigg( \frac{1}{\alpha_1\alpha_2}G_n +\alpha_{13}\Gamma_{2}(G_n)-\Gamma_{3}(G_n) \bigg)^{2}\right)^{\frac{1}{2}}\\
\nonumber&\quad\quad=\bigg(\frac{1}{\alpha_1^{2} \alpha_2^{2}}\mathbb{E}(G_n^2)+ \alpha_{13}^{2}\mathbb{E}\left(\Gamma_{2}^{2}(G_n)  \right)\\
	\nonumber&\quad\quad\quad\quad+\mathbb{E}\left(\Gamma_{3}^{2}(G_n)  \right)+\frac{2\alpha_{13}}{\alpha_1\alpha_2}\mathbb{E}\left(G_n \Gamma_{2}(G_n) \right)\\
	&\quad\quad\quad\quad\quad-\frac{2}{\alpha_1\alpha_2}\mathbb{E}\left(G_n \Gamma_{3}(G_n) \right)-2\alpha_{13}\mathbb{E} \bigg(\Gamma_{2}(G_n)\Gamma_3(G_n)  \bigg)\bigg)^{\frac{1}{2}}.\label{sm2}
\end{align}
\noindent
Using Lemma \ref{gammopformula} in \eqref{sm2}, and rearranging the terms, we get
\begin{align}
\nonumber	&\mathbb{E}\bigg| \frac{1}{\alpha_1\alpha_2}G_n +\alpha_{13}\Gamma_{2}(G_n)-\Gamma_{3}(G_n) \bigg|\\
\nonumber&\quad \leq \bigg\{\frac{1}{120}\kappa_{6}(G_n)-\frac{1}{12}\alpha_{13}\kappa_{5}(G_n)+\left(\frac{1}{6}\alpha_{13}^{2}- \frac{1}{3\alpha_1\alpha_2}\right)\kappa_{4}(G_n)\\
	\nonumber &\quad\quad\quad\quad+\frac{1}{\alpha_1 \alpha_2}\alpha_{13}\kappa_{3}(G_n)+\frac{1}{4}\kappa_{3}^{2}(G_n)-\alpha_{13}\kappa_{2}(G_n)\kappa_{3}(G_n)+\frac{1}{\alpha_1^{2}\alpha_2^{2}}\kappa_{2}(G_n)\\
	&\quad\quad\quad\quad\quad\quad\quad+\alpha_{13}^{2}\kappa_{2}^{2}(G_n)\bigg\}^{1/2}.\label{sm4}
\end{align}

\noindent
Using \eqref{sm4} and \eqref{sm0} in the right hand side of \eqref{sm6}, we get \eqref{sm7}, as desired.
\end{proof}
\noindent
The following corollary establishes the convergence of $G_n$ to a $BG(\alpha_1,p_1,\alpha_2,p_2)$
distribution.
\vfill

\begin{cor}
Let $X\sim BG(\alpha_1,p_1,\alpha_{2},p_2)$ such that $(i)~ \alpha_1\alpha_2>(1+|\alpha_1-\alpha_{2}|) \text{ and }(ii)~ \frac{p_1}{p_2}=\frac{\alpha_1}{\alpha_2}.$ If $\kappa_{j}(G_n) \to  \kappa_{j}(X)=(j-1)!\bigg(\frac{p_1}{\alpha_1^j}+$ $(-1)^j \frac{p_2}{\alpha_2^j}  \bigg)$, $2\leq j \leq 6$, as $n\to\infty$, then 
$$d_3(G_n,X) \to 0, \text{ as } n\to \infty.$$ 
\noindent
That is, $G_n \overset{\mathcal{L}}{\to} X.$
\end{cor}
\begin{proof}
	First note from the assumption (ii) that 
	\begin{align}\label{smtl0}
		\bigg|\frac{p_1}{\alpha_{1}}-\frac{p_2}{\alpha_2}  \bigg|=0 \text{ and }\kappa_{2}(X)=\left( \frac{p_1}{\alpha_1^2} +\frac{p_2}{\alpha_2^2}\right)=\frac{p_1+p_2}{\alpha_1\alpha_2}.
	\end{align}
	\noindent
	Since $\kappa_{j}(G_n) \to  \kappa_{j}(X)=(j-1)!\bigg(\frac{p_1}{\alpha_1^j}+$ $(-1)^j \frac{p_2}{\alpha_2^j}  \bigg)$, $2\leq j \leq 6$, the first quantity, of the bound for $d_3$ in \eqref{sm7}, converges to
		\begin{align}
		\nonumber &  \frac{1}{3} \alpha_{12} \bigg\{\frac{1}{120}\kappa_{6}(X)- \frac{1}{12}\alpha_{13}\kappa_{5}(X)\\
		\nonumber &\quad+\left(\frac{1}{6}\alpha_{13}^{2} -\frac{1}{3\alpha_1 \alpha_2}  \right)\kappa_{4}(X)+\frac{1}{\alpha_1\alpha_2}\alpha_{13}\kappa_{3}(X)+\frac{1}{4}\kappa_{3}^{2}(X)\\
		\nonumber	&\quad\quad- \alpha_{13}\kappa_{2}(X)\kappa_{3}(X)+\frac{1}{\alpha_1^{2} \alpha_2^{2}}\kappa_{2}(X)+\alpha_{13}^{2}\kappa_{2}^{2}(X)    \bigg\}^{1/2}\\
		\nonumber&\quad = \frac{1}{3} \alpha_{12} \bigg\{ \frac{1}{\alpha_1^{2} \alpha_2^{2}} \left(\frac{p_1}{\alpha_1}-\frac{p_2}{\alpha_2}   \right)^{2}  \bigg\}^{\frac{1}{2}}\\
		\nonumber&\quad= \frac{\alpha_{12}}{3\alpha_1 \alpha_2} \bigg(\frac{p_1}{\alpha_1}- \frac{p_2}{\alpha_2} \bigg)\\
		\nonumber&\quad = 0, 
		\end{align}	
		\noindent
		where the last equality follows by the assumption (ii). Using the above fact, together with \eqref{smtl0}, in \eqref{sm7}, the result follows.  
\end{proof}
\noindent
The following corollary immediately follows for $VG(\alpha_1,\alpha_2,p)$, $SVG(\alpha,p)$ and $La(\alpha)$ distributions. 
\begin{cor}\label{keyvgsm}
Consider the sequence $G_n=I_2(f_{n}),~n\geq 1$ such that $f_{n} \in \mathcal{H}^{\odot 2}  $.
\begin{enumerate}
	\item[(i)] Let $X_v\sim VG(\alpha_1,\alpha_2,p)$ such that $\alpha_1\alpha_2>(1+|\alpha_{1}-\alpha_{2}|)$. Let $\alpha_{12}$ and $\alpha_{13}$ be defined in \eqref{alpi}. Then,
		\begin{align}
		\nonumber d_{{3}}(G_{n}, X_v) &\leq \frac{1}{3} \alpha_{12} \bigg(\frac{1}{120}\kappa_{6}(G_n)- \frac{1}{12}\alpha_{13}\kappa_{5}(G_n)\\
		\nonumber &\quad+\left(\frac{1}{6}\alpha_{13}^{2} -\frac{1}{3\alpha_1 \alpha_2}  \right)\kappa_{4}(G_n)+\frac{1}{\alpha_1\alpha_2}\alpha_{13}\kappa_{3}(G_n)+\frac{1}{4}\kappa_{3}^{2}(G_n)\\
		\nonumber	&\quad\quad- \alpha_{13}\kappa_{2}(G_n)\kappa_{3}(G_n)+\frac{1}{\alpha_1^{2} \alpha_2^{2}}\kappa_{2}(G_n)+\alpha_{13}^{2}\kappa_{2}^{2}(G_n)    \bigg)^{1/2}\\ &\quad\quad\quad\quad+\frac{1}{2} \alpha_{12} \left|\frac{2p}{\alpha_1\alpha_2}-\kappa_{2}(G_n)  \right|+p\alpha_{12}|\alpha_{13}| \label{sm07}
		\end{align}
		
	\item [(ii)] Let $X_s\sim SVG(\alpha,p)$. Then, for $\alpha>1,$
	\begin{align}
	\nonumber	d_{{3}}(G_{n}, X_s)& \leq \frac{\alpha^{2}}{3(\alpha^{2}-1)}\left( \frac{1}{120}\kappa_{6}(G_n)-\frac{1}{3\alpha^{2}}\kappa_{4}(G_{n})+\frac{1}{4}\kappa_{3}^{2}(G_n)+\frac{1}{\alpha^{4}}\kappa_{2}(G_{n})  \right)^{\frac{1}{2}}\\
	&\quad\quad+\frac{\alpha^{2}}{2(\alpha^{2}-1)}\left|\frac{2p}{\alpha^{2}}-\kappa_{2}(G_{n})   \right|.\label{svgsm}
	\end{align}

	\item [(iii)] Let $X_l\sim La(\alpha)$. Then, for $\alpha>1$,
	\begin{align}
	\nonumber	d_{{3}}(G_{n}, X_l)& \leq \frac{\alpha^{2}}{3(\alpha^{2}-1)}\left( \frac{1}{120}\kappa_{6}(G_n)-\frac{1}{3\alpha^{2}}\kappa_{4}(G_{n})+\frac{1}{4}\kappa_{3}^{2}(G_n)+\frac{1}{\alpha^{4}}\kappa_{2}(G_{n})  \right)^{\frac{1}{2}}\\
	&\quad\quad+\frac{\alpha^{2}}{2(\alpha^{2}-1)}\left|\frac{2}{\alpha^{2}}-\kappa_{2}(G_{n})   \right|.\label{lasm}
	\end{align}		
\end{enumerate}
\end{cor}

\begin{rem}
(i) Eichelsbacher and Th$\ddot{\text{a}}$le \cite{key2} first obtained bounds in terms of cumulants of order two to six for the Wasserstein distance $d_W$ for the centered $VG$, $SVG$ and Laplace approximations of  $G_n=I_2(f_{n})$, $n\geq 1$, where $f_{n} \in \mathcal{H}^{\odot 2}$.  However, their bounds contain some indeterminate positive constants.

\vskip 1ex
\item [(ii)]Recently, Gaunt \cite{kk2,gauntnew} established upper bounds in the $d_W$ distance for $SVG$ and $VG$ approximations to $G_n$,with explicit constants. In particular, for the Laplace case, when $\kappa_{2}(G_n) = \frac{2}{\alpha^2}$, Gaunt \cite[Corollary 5.4]{kk2} obtained the following bound, for $\alpha>0,$ as
\begin{align}
d_{W}(G_{n}, X_l)& \leq 3\alpha^2\left( \frac{1}{120}\kappa_{6}(G_n)-\frac{1}{3\alpha^{2}}\kappa_{4}(G_{n})+\frac{1}{4}\kappa_{3}^{2}(G_n)+\frac{1}{\alpha^{4}}\kappa_{2}(G_{n})  \right)^{\frac{1}{2}}.\label{lasm2}
\end{align}	
\noindent
When $\kappa_{2}(G_n) = \frac{2}{\alpha^2}$, our bound for $\alpha>1$, is

\begin{align}
d_{{3}}(G_{n}, X_l)& \leq \frac{\alpha^{2}}{3(\alpha^{2}-1)}\left( \frac{1}{120}\kappa_{6}(G_n)-\frac{1}{3\alpha^{2}}\kappa_{4}(G_{n})+\frac{1}{4}\kappa_{3}^{2}(G_n)+\frac{1}{\alpha^{4}}\kappa_{2}(G_{n})  \right)^{\frac{1}{2}}.\label{lasm1}
\end{align}	
\noindent
Note that when $\alpha>1.054$, $3\alpha^2>\frac{\alpha^2}{3(\alpha^2-1)}$, and the bound in \eqref{lasm1} is sharper than the one in \eqref{lasm2}. Moreover, if $\mathbb{E}(G_n^j) \to \mathbb{E}(X_l^j)=\frac{j!}{\alpha^j}$, for $j=2,4,6,$ then $G_n \overset{\mathcal{L}}{\to}X_l.$ This result was proved by Eichelsbacher and Th$\ddot{\text{a}}$le \cite{key2}. 
\end{rem}
\noindent
Next we present the result for normal distribution similar to Corollary \ref{norgam0} and establish the convergence of $G_n$ to a normal $\mathcal{N}(0,\alpha^2)$ distribution.
\begin{cor}\label{corng}
Consider the sequence $G_n=I_2(f_{n}),~n\geq 1$ such that $f_{n} \in \mathcal{H}^{\odot 2}  $. Let $Z_\alpha \sim \mathcal{N}(0,\alpha^2)$. Then, for $\alpha>1$,
		\begin{align}
			\nonumber d_3(G_n,Z_\alpha) &\leq \frac{1}{3} \bigg(\frac{1}{120}\kappa_{6}(G_n)+\frac{1}{4}\kappa_{3}^{2}(G_n)  \bigg)^{\frac{1}{2}}+\frac{1}{2}|\alpha^2-\kappa_{2}(G_n)|.\label{corng1}
		\end{align}
		\noindent
	Furthermore, when $\kappa_{2}(G_n) \to \alpha^2$, and $\kappa_j(G_n) \to 0$, for $j=3,6$, we have $G_n \overset{\mathcal{L}}{\to} Z_\alpha.$
\end{cor}

\begin{proof}
	 Note that the first term (up to a constant) in \eqref{norgam1}
	\begin{align}
		\mathbb{E}|\Gamma_{3}(G_n)| &\leq \bigg( \mathbb{E}(\Gamma_{3}^{2}(G_n)) \bigg)^{\frac{1}{2}}\\
		\nonumber&\quad\quad\quad\quad \text{(by Cauchy Schwartz inequality)}\\
		&=\bigg(\frac{1}{120}\kappa_{6}(G_n)+\frac{1}{4}\kappa_{3}^{2}(G_n)  \bigg)^{\frac{1}{2}},\label{norsm1}
	\end{align}
	\noindent
	where the last equality follows by \eqref{gammopformula2}. Also, it is known that
	\begin{align}
		\kappa_{2}(Z_\alpha)=\alpha^2 ,~\kappa_{1}(G_n)=\mathbb{E}(G_n)=0\text{ and } \mathbb{E}(\Gamma_{2}(G_n))=\kappa_{2}(G_n).\label{norsm2}
	\end{align}
	\noindent
	Using \eqref{norsm1} and \eqref{norsm2} in \eqref{norgam1}, we get
	\begin{align*}
 d_3(G_n,Z_\alpha) &\leq \frac{1}{3} \bigg(\frac{1}{120}\kappa_{6}(G_n)+\frac{1}{4}\kappa_{3}^{2}(G_n)  \bigg)^{\frac{1}{2}}+\frac{1}{2}|\alpha^2-\kappa_{2}(G_n)|.
	\end{align*}
	\noindent
This proves the result.	 
\end{proof}

\section{Bilateral gamma approximation to homogeneous sums}\label{applica}
\noindent
In this section, we establish a new limit theorem for $BG$ distributions of homogeneous sums using our results. The limit theorems related to homogeneous sums have been discussed by Eichelsbacher and Th$\ddot{\text{a}}$le \cite{key2}, Mossel {\it et al.} \cite{MOO0}, Nourdin and Peccati \cite{nourdin} and Nourdin {\it et al.} \cite{npr0}. First, we give the definition of homogeneous sums (see \cite{npr0} or \cite{nourdin}).
\begin{defn}[Homogeneous sums]\label{HS0}
	Let $[N]:=\{ 1,\ldots,N \}$, where $N\geq 2$, and $Y=\{ Y_i\}_{i\geq 1}$ be a sequence of centered independent rvs. For an integer $1 \leq q \leq N$, let $f: [N]^q \to \mathbb{{R}}$ be a symmetric function vanishing on diagonals, that is, $f(i_1,\ldots,i_q)=0$, whenever $i_k=i_j$, for some $k\neq j$. The rv
	\begin{align}
	\nonumber	H_q(N,f,Y)&=\sum_{1 \leq i_1,\ldots,i_q\leq N}f(i_1,\ldots,i_q)Y_{i_1}\ldots Y_{i_q}\\
	\nonumber	&=q!\sum_{ \{ i_1,\ldots,i_q\} \subset [N]^{q} }f(i_1,\ldots,i_q)Y_{i_1}\ldots Y_{i_q}\\
		&=q!\sum_{  1\leq i_1<\ldots<i_q\leq N  }f(i_1,\ldots,i_q)Y_{i_1}\ldots Y_{i_q}\label{HoSumsDef}
	\end{align}
	is called the homogeneous sum of order $q$, based on $f$, and on the first $N$ elements of $Y$.
\end{defn}

\noindent
Note that $\mathbb{E}(H_q(N,f,Y))=0 $, and if $\mathbb{E}(Y_i^2)=1$, for all $i\in \mathbb{N}$, then $\mathbb{E}(H_q^2(N,f,Y))= q! \| f\|_{q}^{2},$ where $\| f\|_{q}^{2}=\sum_{  1\leq i_1<\ldots<i_q\leq N  }f^{2}(i_1,\ldots,i_q).$

\vskip 1ex
\noindent
Next recall that (see \cite[p.184]{nourdin}), the quantity
\begin{align}
	\text{Inf}_i(f)&:=\sum_{ \{ i_2,\ldots,i_q\} \subset [N]^{q-1} }f^2(i,i_2,\ldots,i_q)\\
	&=\frac{1}{(q-1)!}\sum_{  1\leq i_2<\ldots<i_q\leq N  }f^2(i,i_2,\ldots,i_q)
\end{align}
\noindent
is called the influence of the variable $Y_i$. It is known that (see equation (11.2.5) of \cite{nourdin}) if $\{Y_i\}_{i\geq 1}$ is a sequence of independent rvs with $\mathbb{E}(Y_i)=0$ and $\mathbb{E}(Y_i^2)=1$ for all $i$, then
\begin{align}\label{inf0}
	\mathbb{E}\bigg[\text{Var}\bigg(H_q(N,f,Y) \mid Y_k, k\neq i \bigg)  \bigg]=(q!)^2 \text{Inf}_i(f),
\end{align}
\noindent
which quantifies the impact of the rv $Y_i$ on the overall fluctuations of $H_q(N,f,Y)$.

\vskip 1ex
\noindent
 Next let $Z=\{Z_i \}_{i\geq 1}$ be a sequence of independent and identically distributed (i.i.d.) normal $\mathcal{N}(0,1)$ rvs. Then from Remark 11.2.5 of \cite{nourdin}, we get
 \begin{align}\label{hmsum1}
 	H_q(N,f,Z)=\sum_{1 \leq i_1,\ldots,i_q\leq N}f(i_1,\ldots,i_q)Z_{i_1}\ldots Z_{i_q}=I_q(g),
 \end{align}
\noindent
where $I_q$ denotes a $q$th order multiple stochastic integral (see \eqref{expansion}) and $g$ is given by
\begin{align}
	g=\sum_{1 \leq i_1,\ldots,i_q\leq N}f(i_1,\ldots,i_q)e_{i_1}\otimes e_{i_2}\otimes \ldots \otimes e_{i_q} \in \mathcal{H}^{\odot q}.
\end{align} 

\noindent
Also, $\mathbb{E}(H_q(N,f,Z))=0$, and using \eqref{prductfor2} in \eqref{hmsum1}, we get $$\mathbb{E}\bigg( H_q^2(N,f,Z)\bigg)=q! \|g\|_{\mathcal{H}^{\otimes q} }^{2}.$$
\noindent
Next we state a lemma, which measures closeness between two distributions of homogeneous sums that are built from different sequence of independent rvs. The following lemma is essentially part 3 of Theorem 4.1 of \cite{npr0}.

\begin{lem}\label{HMSumk}
	Let $Y=\{Y_i\}_{i\geq 1}$ be a sequence of centered i.i.d. rvs and $Z=\{Z_i\}_{i\geq 1}$ be a sequence of i.i.d. normal $\mathcal{N}(0,1)$ rvs. Fix $q\geq 1$, and let $\{N_n,f_n \}_{ n\geq 1}$ be a sequence such that $\{N_n \}_{n\geq 1 }$ is an integer sequence going to infinity, and each $f_n:[N]^q \to \mathbb{R}$ is a symmetric function that vanishes on diagonals. Then
	\begin{align}\label{homosumlimk}
	d_3(H_q(N_n,f_n,Y),H_q(N_n,f_n,Z)) \leq q!(30 \rho)^q  \sqrt{\max_{1\leq i \leq N_n} \text{Inf}_i (f_n)} ,
	\end{align}	
	where $\rho :=\sup_{i \geq 1} \mathbb{E}[ |Y_i|^3] < \infty$.  
\end{lem}
\begin{proof}
 Since by part 3 of Theorem 4.1 of \cite{npr0}, we have 
\begin{align*}
|\mathbb{E}(h(H_q(N_n,f_n,Y)))-\mathbb{E}(h(H_q(N_n,f_n,Z)))| \leq \|h^{(3)} \| (30 \rho)^q q! \sqrt{\max_{1\leq i \leq N_n} \text{Inf}_i (f_n)},
\end{align*}
\noindent
where $h:\mathbb{{R}}\to \mathbb{{R}}$ is a thrice differentiable function such that $\|h^{(3)} \| < \infty$. Hence for every $h \in \mathcal{W}_3,$ (see \eqref{fs1}), the bound \eqref{homosumlimk} holds.	
\end{proof}

\noindent
Next, we establish an another important lemma, which is helpful for Theorem \ref{CtoHSk}. Let us first define $\widetilde{\kappa}_j:=\kappa_j(G_n)-\kappa_j(X)$, $\kappa_{ij}(G_n):=\kappa_i(G_n)\kappa_j(G_n)$ and $\kappa_{ij}(X):=\kappa_i(X)\kappa_j(X)$, $i,j\geq 2$.
\begin{lem}\label{lemmarp}
	Consider the sequence $G_n=I_2(g_{n})$, $n\geq 1$, where $g_{n} \in \mathcal{H}^{\odot 2}$. Let $X\sim BG(\alpha_1,p_1,\alpha_{2},p_2)$ such that $\alpha_1\alpha_2>(1+|\alpha_1-\alpha_{2}|)$ and $p_1\alpha_2=p_2\alpha_1$ (so that $\mathbb{E}(X)=0$). Then, 
	\begin{align}
	\nonumber(i)~~ d_{{3}}(G_{n}, X) &\leq \frac{1}{3} \alpha_{12} \bigg(\frac{1}{\sqrt{120}}\sqrt{|\widetilde{\kappa}_{6}|}+ \frac{\sqrt{|\alpha_{13}|}}{2\sqrt{3}}\sqrt{|\widetilde{\kappa}_{5}|}+\sqrt{\left|\frac{1}{6}\alpha_{13}^{2} -\frac{1}{3\alpha_1 \alpha_2}  \right|}\sqrt{|\widetilde{\kappa}_{4}|}\\
	\nonumber &\quad+\frac{\sqrt{|\alpha_{13}|}}{\sqrt{\alpha_1\alpha_2}}\sqrt{|\widetilde{\kappa}_{3}|}+\frac{1}{2}\sqrt{|\widetilde{\kappa}_{3}|}+\frac{1}{\sqrt{2}}\sqrt{|\widetilde{\kappa}_{3}\kappa_{3}(X)|}\\
	\nonumber	&\quad\quad\quad+ \sqrt{|\alpha_{13}|}\sqrt{|\kappa_{23}(G_n)-\kappa_{23}(X)|}+\frac{1}{\alpha_1 \alpha_2}\sqrt{|\widetilde{\kappa}_{2}|}\\
	&\quad\quad\quad\quad+|\alpha_{13}||\widetilde{\kappa}_{2}| + \sqrt{2}|\alpha_{13}|\sqrt{|\widetilde{\kappa}_{2}\kappa_{2}(X) |}  \bigg)+\frac{1}{2}\alpha_{12} \left|\widetilde{\kappa}_{2}  \right|, \label{rp01}
	\end{align}
	\noindent
	where $\alpha_{12}$ and $\alpha_{13}$ are defined in \eqref{alpi}.
	
	\vskip 1ex
	\noindent
	$(ii)$ Furthermore, when $\kappa_{j}(G_n) \to \kappa_j(X)$, for $2\leq j \leq 6$, $d_3(G_n,X) \to 0, \text{ as }n\to \infty.$ That is, $G_n \overset{L}{\to} X.$
\end{lem}
\begin{proof}
	First note that
	\begin{align}
	\nonumber\left|\frac{p_1+p_2}{\alpha_1\alpha_2} -\kappa_{2}(G_n)  \right|&=\left|\frac{p_1+p_2}{\alpha_1\alpha_2} -(\kappa_{2}(G_n)-\kappa_{2}(X))-\kappa_{2}(X)  \right|\\
	\nonumber&=\left|\kappa_{2}(G_n)-\kappa_{2}(X)  \right|~~ \bigg( \text{since } \kappa_{2}(X)=\frac{p_1+p_2}{\alpha_1\alpha_2}\bigg)\\
	&=|\widetilde{\kappa}_{2}|	.\label{rp002}
	\end{align}
	\noindent
	Recall (see \eqref{cuforbg}) that, for $j\geq1$,
	\begin{align}\label{rp003}
	\kappa_j(X)=(j-1)!\left(\frac{p_1}{\alpha_1^j}+(-1)^j \frac{p_2}{\alpha_2^j}  \right).
	\end{align}
	\noindent
	Consider next the term inside the curly bracket of \eqref{sm7}. That is, let 
	\begin{align}
	\nonumber\Delta_n:=&\frac{1}{120}\kappa_{6}(G_n)- \frac{1}{12}\alpha_{13}\kappa_{5}(G_n)\\
	\nonumber &\quad+\left(\frac{1}{6}\alpha_{13}^{2} -\frac{1}{3\alpha_1 \alpha_2}  \right)\kappa_{4}(G_n)+\frac{1}{\alpha_1\alpha_2}\alpha_{13}\kappa_{3}(G_n)+\frac{1}{4}\kappa_{3}^{2}(G_n)\\
	\nonumber	&\quad\quad- \alpha_{13}\kappa_{2}(G_n)\kappa_{3}(G_n)+\frac{1}{\alpha_1^{2} \alpha_2^{2}}\kappa_{2}(G_n)+\alpha_{13}^{2}\kappa_{2}^{2}(G_n)  \\
	\nonumber	=&\bigg(\frac{1}{120}\widetilde{\kappa}_{6}- \frac{1}{12}\alpha_{13}\widetilde{\kappa}_{5}+\left(\frac{1}{6}\alpha_{13}^{2} -\frac{1}{3\alpha_1 \alpha_2}  \right)\widetilde{\kappa}_{4}\\
	\nonumber &\quad+\frac{1}{\alpha_1\alpha_2}\alpha_{13}\widetilde{\kappa}_{3}+\frac{1}{4}\widetilde{\kappa}_{3}^{2}+\frac{1}{2}\kappa_{3}(X)\widetilde{\kappa}_3- \alpha_{13}\bigg(\kappa_{23}(G_n)-\kappa_{23}(X)\bigg)\\
	\nonumber	&\quad\quad+\frac{1}{\alpha_1^{2} \alpha_2^{2}}\widetilde{\kappa}_{2}+\alpha_{13}^{2}\widetilde{\kappa}_{2}^{2}+2\alpha_{13}^{2}\kappa_{2}(X)\widetilde{\kappa}_{2} \bigg)+ \bigg\{\frac{1}{120}\kappa_{6}(X)- \frac{1}{12}\alpha_{13}\kappa_{5}(X)\\
	\nonumber &\quad\quad\quad\quad\quad+\left(\frac{1}{6}\alpha_{13}^{2} -\frac{1}{3\alpha_1 \alpha_2}  \right)\kappa_{4}(X)+\frac{1}{\alpha_1\alpha_2}\alpha_{13}\kappa_{3}(X)+\frac{1}{4}\kappa_{3}^{2}(X)\\
	&\quad\quad\quad\quad\quad\quad- \alpha_{13}\kappa_{2}(X)\kappa_{3}(X)+\frac{1}{\alpha_1^{2} \alpha_2^{2}}\kappa_{2}(X)+\alpha_{13}^{2}\kappa_{2}^{2}(X)\bigg\} ,\label{curly}
	\end{align}
	
	\noindent
	where the last equality follows by rearranging the terms and using the fact $a^2=(a-b)^2+2b(a-b)+b^2.$ By \eqref{rp003}, we get the terms inside the curly bracket of \eqref{curly} vanishes, and so
	\begin{align}
	\nonumber	\Delta_n=&\frac{1}{120}\widetilde{\kappa}_{6}- \frac{1}{12}\alpha_{13}\widetilde{\kappa}_{5}\\
	\nonumber&\quad+\left(\frac{1}{6}\alpha_{13}^{2} -\frac{1}{3\alpha_1 \alpha_2}  \right)\widetilde{\kappa}_{4}+\frac{1}{\alpha_1\alpha_2}\alpha_{13}\widetilde{\kappa}_{3}+\frac{1}{4}\widetilde{\kappa}_{3}^{2}\\
	\nonumber	&\quad\quad+\frac{1}{2}\kappa_{3}(X)\widetilde{\kappa}_3- \alpha_{13}\bigg(\kappa_{23}(G_n)-\kappa_{23}(X)\bigg)\\
	\label{rp004}&\quad\quad\quad+\frac{1}{\alpha_1^{2} \alpha_2^{2}}\widetilde{\kappa}_{2}+\alpha_{13}^{2}\widetilde{\kappa}_{2}^{2}+2\alpha_{13}^{2}\kappa_{2}(X)\widetilde{\kappa}_{2}. 
	\end{align} 
	\noindent
	Substituting \eqref{rp002} and \eqref{rp004} in \eqref{sm7}, and then using the triangle inequality, we get \eqref{rp01}, as desired.
\end{proof}
\begin{rem}
	From \eqref{rp01}, if $\kappa_j(G_n) \to \kappa_j(X),~j=2,3,4$, then
	\begin{align}
	\nonumber d_{{3}}(G_{n}, X) &\leq \frac{1}{3} \alpha_{12} \bigg(\frac{1}{\sqrt{120}}\sqrt{|\widetilde{\kappa}_{6}|}+ \frac{\sqrt{|\alpha_{13}|}}{2\sqrt{3}}\sqrt{|\widetilde{\kappa}_{5}|}\bigg), 
	\end{align}
	\noindent
	where $\alpha_{12}$ and $\alpha_{13}$ are defined in \eqref{alpi}.
\end{rem}
\noindent
Using the above two lemmas, we can now deduce an upper bound for $BG$ approximation to homogeneous sums of independent rvs.
\begin{thm}\label{CtoHSk}
	Let $\{N_n,f_n  \}_{n\geq 1}$, $Y=\{Y_i\}_{i\geq 1}$ be the sequences as stated in Lemma \ref{HMSumk} and $H_2(N_n,f_n,Y)$ be defined as in \eqref{HoSumsDef}. Consider the sequence $G_n=I_2(g_{n})$, $n\geq 1$, where $g_{n}=\sum_{1 \leq i_1,i_2\leq N_n}f_n(i_1,i_2)e_{i_1}\otimes e_{i_2}\in \mathcal{H}^{\odot 2}$. Let $X\sim BG(\alpha_1,p_1,\alpha_{2},p_2)$ with $\alpha_1\alpha_2>(1+|\alpha_1-\alpha_{2}|)$ and $p_1\alpha_2=p_2\alpha_1$ (so that $\mathbb{E}(X)=0$). If $\rho =\sup_{i \geq 1} \mathbb{E}[ |Y_i|^3] < \infty$, then
	\begin{align}
\nonumber	d_3(H_2(N_n,f_n,Y), X) &\leq 2(30 \rho)^2 \sqrt{\max_{1\leq i \leq N_n} \text{Inf}_i (f_n)}+\frac{1}{3} \alpha_{12} \bigg(\frac{\sqrt{|\widetilde{\kappa}_{6}|}}{\sqrt{120}}+ \frac{\sqrt{|\alpha_{13}|}}{2\sqrt{3}}\sqrt{|\widetilde{\kappa}_{5}|}\\
\nonumber &\quad\quad+\sqrt{\left|\frac{1}{6}\alpha_{13}^{2} -\frac{1}{3\alpha_1 \alpha_2}  \right|}\sqrt{|\widetilde{\kappa}_{4}|}+\frac{\sqrt{|\alpha_{13}|}}{\sqrt{\alpha_1\alpha_2}}\sqrt{|\widetilde{\kappa}_{3}|}\\
\nonumber&\quad\quad\quad+\frac{1}{2}\sqrt{|\widetilde{\kappa}_{3}|}+\frac{1}{\sqrt{2}}\sqrt{|\kappa_{3}(X) \widetilde{\kappa}_{3}|}\\
\nonumber	&\quad\quad\quad\quad+ \sqrt{|\alpha_{13}|}\sqrt{|\kappa_{23}(G_n)-\kappa_{23}(X)|}+\frac{1}{\alpha_1 \alpha_2}\sqrt{|\widetilde{\kappa}_{2}|}\\
&\quad\quad\quad\quad\quad+|\alpha_{13}||\widetilde{\kappa}_{2}| + \sqrt{2}|\alpha_{13}|\sqrt{|\kappa_{2}(X) \widetilde{\kappa}_{2}|}  \bigg)+\frac{1}{2}\alpha_{12} \left|\widetilde{\kappa}_{2}  \right|,\label{newhombd}
	\end{align} 
  \noindent
  where $\alpha_{12}$ and $\alpha_{13}$ are defined in \eqref{alpi}. 
\end{thm}
\begin{proof}
	Let $Z=\{Z_i\}_{i\geq 1}$ be a sequence of i.i.d. normal $\mathcal{N}(0,1)$ rvs and $H_2(N_n,f_n,Z)$ be defined as Definition \ref{HS0}. Then by the triangle inequality, we have
	\begin{align}
	\nonumber	d_3(H_2(N_n,f_n,Y), X) &\leq d_3(H_2(N_n,f_n,Y), H_2(N_n,f_n,Z))+d_3(H_2(N_n,f_n,Z), X)\\
	\label{hmsum0k}	&=d_3(H_2(N_n,f_n,Y), H_2(N_n,f_n,Z))+d_3(I_2(g_n), X),
	\end{align}
	\noindent
	where the last term of \eqref{hmsum0k} is followed by \eqref{hmsum1}. Using \eqref{homosumlimk} and \eqref{rp01} in \eqref{hmsum0k}, we get \eqref{newhombd}, as desired.
\end{proof}
\noindent
The following corollary gives a new limit theorem for homogeneous sums so that the limiting distribution belongs to the class of $BG$ distributions.
\begin{cor}\label{CtoHS}
Assume the conditions of Theorem \ref{CtoHSk} hold. If (i) $\max_{1\leq i \leq N_n} \text{Inf}_i (f_n) \to 0$, and (ii) $\kappa_j(G_n) \to \kappa_j(X),~2\leq j\leq 6$, then

	\begin{align}
	d_3(H_2(N_n,f_n,Y), X) \to 0 \text{ as } n\to \infty.
	\end{align}
	That is, $H_2(N_n,f_n,Y)\overset{L}{\to} X,$ as $n \to \infty$.  
\end{cor}
\noindent
The proof is immediate from \eqref{newhombd} as $\widetilde{\kappa}_j \to 0,$ for $\kappa_j(G_n) \to \kappa_j(X),~2\leq j\leq 6$.

\vskip 1ex
\noindent
As corollaries, we also have the following results for $SVG(\alpha,p)$ and $\mathcal{N}(0,\alpha^2)$ distributions. Let $\max_{1\leq i \leq N_n} \text{Inf}_i (f_n) \to 0$, as $n \to \infty.$ The following limit theorems hold:
\begin{enumerate}
	\item [(i)] If $\kappa_j(G_n) \to \kappa_j(X_s),~j=2,4,6$, then $H_2(N_n,f_n,Y) \overset{L}{\to} X_s, \text{ as }n\to\infty,$ where $X_s \sim SVG(\alpha,p)$, $\alpha>1$, $p>0$.
	
	\item [(ii)] If $\kappa_{2}(G_n) \to \alpha^2$, and $\kappa_j(G_n) \to 0$, for $j=3,6$, then $H_2(N_n,f_n,Y) \overset{L}{\to} Z_\alpha,  \text{ as }n\to\infty,$ where $Z_\alpha \sim \mathcal{N}(0,\alpha^2)$, $\alpha>1$.
\end{enumerate}

\begin{rem}
	(i) Note that the limit theorem for $SVG(\alpha,p)$ in (i) above was first proved by Eichelsbacher and Th$\ddot{\text{a}}$le \cite{key2}.
	
	\vskip 1ex
	\item [(ii)] It is shown in Theorem 1.9 of \cite{npr0} that
 $$\kappa_{4}(G_n)\to 0 \Rightarrow H_2(N_n,f_n,Y) \overset{L}{\to}Z_\alpha, \text{ as } n\to \infty,$$ 
 which is stronger result than our result (ii) above. 
\end{rem}
\noindent
Finally, as an application of Theorem \ref{CtoHSk}, we discuss a limit theorem for a second order $U$-statistic.
\begin{exmp}[Bilateral-gamma approximation to $U$-statistic]
	Let $Y=\{Y_i\}_{i\geq 1}$ be a sequence of centered i.i.d. rvs such that $\rho = \sup_{i \geq 1}\mathbb{E}[ |Y_i|^3]< \infty$. Following Section 5.1.1 of Examples (ii) of \cite{serfling}, we consider a second order $U$-statistic as
	\begin{align}\label{ustat}
		U_n=\frac{1}{\binom{n}{2}} \sum_{1 \leq i_1 <i_2\leq n}Y_{i_1}Y_{i_2},~n>1.
	\end{align}
\noindent
Note that \eqref{ustat} can be viewed as homogeneous sum $H_2(n,f_n,Y)$ (see \eqref{HoSumsDef}). Here,
\begin{align*}
	f_n(i_1,i_2)&= \frac{1}{2 \binom{n}{2}}\mathbf{1}_{ \{1 \leq i_1 <i_2\leq n\} }\\
	&=\frac{1}{n(n-1)}\mathbf{1}_{ \{1 \leq i_1 <i_2\leq n \}}.
\end{align*}
 \noindent
 Following steps similar to Example 11.2.7 of \cite{nourdin}, it follows that
 \begin{align*}
 	\text{Inf}_{i}(f_n)&=\frac{(n-1)}{(2!)^2 (2-1) n(n-1)}\\
 	&=\frac{1}{4n}.
 \end{align*}
 \noindent
 So, $\text{Inf}_i(f_n) \to 0$ as $n\to \infty$. Next, let $Z_i=I_1(e_i)$, $i\geq 1$, be a sequence of i.i.d. normal $\mathcal{N}(0,1)$ rvs over an isonormal Gaussian process $W$, where $I_1$ is the Wiener-It$\hat{\text{o}}$ integral of order $1$ and $(e_i)_{i\geq 1}$ is an orthonormal basis of the corresponding separable Hilbert space $\mathcal{H}$. From \cite[p. 2363]{rosen3}, we consider a second order $U$-statistic with degeneracy of order $1$ such that, for all $n\geq 1$,
 \begin{align*}
 	U_{n}^{\star}&=\frac{1}{\binom{n}{2}} \sum_{1 \leq i_1 <i_2\leq n}Z_{i_1}Z_{i_2}\\
 	&=I_2(g_n),
 \end{align*}
 \noindent
 where $g_n=\frac{1}{\binom{n}{2}}e_{i_1}\otimes e_{i_2}\in \mathcal{H}^{\odot 2}$ and $I_2$ denotes the second order Wiener-It$\hat{\text{o}}$ integral. Also, let $X\sim BG(\alpha_1,p_1,\alpha_{2},p_2)$ with $\alpha_1\alpha_2>(1+|\alpha_1-\alpha_{2}|)$ and $p_1\alpha_2=p_2\alpha_1$ (so that $\mathbb{E}(X)=0$). Then, by Theorem \ref{CtoHSk}, if $\kappa_j(U_n^\star) \to \kappa_j (X),$ $2\leq j \leq 6$,
 \begin{align*}
 	d_3(U_n,X) \to 0, \text{ as }n\to \infty.
 \end{align*}
 That is, $U_n \overset{L}{\to}X,$ as $n \to \infty.$

\end{exmp}


\setstretch{1}

\end{document}